\documentclass[a4paper,11pt]{article}
\pdfoutput=1
\usepackage{arxiv}

\usepackage{comment}
\usepackage[T1]{fontenc}
\usepackage[ascii]{inputenc}
\usepackage[english]{babel}
\usepackage{tikz-cd}

\usepackage{pdfsync}

\usepackage{cite}
\usepackage{hyperref,url,nameref,doi}

\usepackage{amsmath,amsfonts,amsthm,amssymb}
\usepackage{todonotes}

\usepackage{booktabs}       %
\usepackage{nicefrac}       %
\usepackage{microtype}      %
\usepackage{setspace}
\allowdisplaybreaks[2]

\allowdisplaybreaks

\newcommand{\R}{\mathbb{R}}
\newcommand{\N}{\mathbb{N}}

\newcommand{\Rn}{\R^n}
\newcommand{\Sn}{\mathbb{S}^{n-1}}
\renewcommand{\O}{\Omega}
\newcommand{\Od}{\O_{\delta}}
\newcommand{\Wd}{\omega_{\delta}}
\newcommand{\dx}{\,\mathrm{d}x}

\newcommand{\grad}{\ddot{\mathcal{G}}_{\delta}}
\newcommand{\diver}{\mathcal{D}_{\delta}}

\newcommand{\U}{\mathbb{U}}
\newcommand{\Uz}{\U_0}

\newcommand{\Udz}{\U_{\delta,0}}

\newcommand{\Q}{\mathbb{Q}}
\newcommand{\Qd}{\mathbb{Q}_{\delta}}

\newcommand{\knl}{{\ddot{\kappa}}}
\newcommand{\qnl}{{\ddot{\sigma}}}

\newcommand{\qnldf}{{\ddot{\sigma}_{\delta,f}}}
\newcommand{\pnl}{{\ddot{\tau}}}

\newcommand{\Id}{{\hat{I}}}

\newcommand{\Idloc}{\Id_{\text{loc}}}

\newcommand{\kadm}{\mathbb{A}}
\newcommand{\kadmd}{\kadm_{\delta}}

\newcommand{\wto}{\rightharpoonup}

\newtheorem{theorem}{Theorem}[section]
\newtheorem{lemma}[theorem]{Lemma}
\newtheorem{proposition}[theorem]{Proposition}
\newtheorem{corollary}[theorem]{Corollary}
\theoremstyle{remark}

\theoremstyle{definition}

\DeclareMathOperator{\graph}{graph}
\DeclareMathOperator{\trace}{tr}

\DeclareMathOperator{\Div}{div}
\DeclareMathOperator{\supp}{supp}

\DeclareMathOperator{\sign}{sign}

\DeclareMathOperator*{\argmin}{arg\,min}


\allowdisplaybreaks
\title{The dual approach to optimal control in the coefficients of nonlocal nonlinear diffusion}

\author{
Marcus Schytt\\
Department of Mathematical Sciences\\
Aalborg University\\
DK--9210 Aalborg, Denmark\\
\And
\href{https://orcid.org/0000-0002-3987-7745}{\includegraphics[scale=0.06]{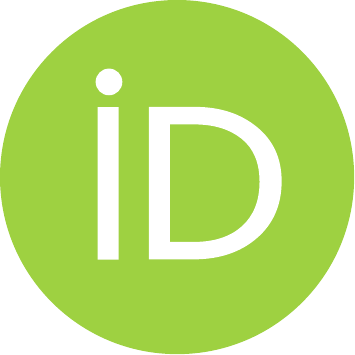}\hspace{1mm}Anton Evgrafov}\\
Department of Mathematical Sciences\\
Aalborg University\\
DK--9210 Aalborg, Denmark\\
\texttt{anev@math.aau.dk}
}

\renewcommand{\shorttitle}{The dual approach to optimal control of nonlocal nonlinear diffusion}

\hypersetup{
pdftitle={The dual approach to optimal control in the coefficients of nonlocal nonlinear diffusion},
pdfauthor={Marcus Schytt, Anton Evgrafov},
pdfkeywords={Optimal design, nonlocal diffusion, mixed variational principles, \(p\)-Laplacian},
pdfsubject={MSC2010 classification: 49J21, 49J45, 49J35, 80M50},
}

\begin{document}
\maketitle

\begin{abstract}
    We derive the dual variational principle (principle of minimal complementary energy) for the nonlocal nonlinear scalar diffusion problem, which may be viewed as the nonlocal version of the \(p\)-Laplacian operator.
    We establish existence and uniqueness of solutions (two-point fluxes) as well as their quantitative stability, which holds uniformly with respect to the small parameter (nonlocal horizon) characterizing the nonlocality of the problem.
    We then focus on the nonlocal analogue of the classical optimal control in the coefficient problem associated with the dual variational principle, which may be interpreted as that of optimally distributing a limited amount of conductivity in order to minimize the complementary energy.
    We show that this nonlocal optimal control problem \(\Gamma\)-converges to its local counterpart, when the nonlocal horizon vanishes.
\end{abstract}





\section{Introduction}\label{sec:intro}

More than half a century ago C{\'e}a and Malanowski~\cite{cea1970example} introduced an optimization problem, which can be viewed as that of distributing a limited amount of ``conductivity'' in a given design domain in order to minimize the  weighted average steady-state temperature for a given volumetric heat source.
This seminal paper catalyzed the development of what is nowadays referred to as topology optimization, see for example the monographs~\cite{cherkaev2012variational,allaire2012shape,bendsoe2013topology} and the references therein.
The original problem has been generalized and extended in countless ways.
We will state it here almost in its original form, except we will assume that the linear diffusion governed by the Laplace operator is replaced by a nonlinear diffusion governed by the  \(p\)-Laplace operator, \(1<p<+\infty\).
Namely, let \(\O\) be a connected bounded open Lipschitz domain in \(\R^n\), \(n \geq 2\), and \(f \in L^q(\O)\) with \(p^{-1}+q^{-1}=1\) be a given volumetric heat source.
We would like to find the optimal distribution of the conductivity \(\kappa \in \kadm \subset L^\infty(\O)\) and the corresponding temperature field \(u \in \Uz = W^{1,p}_0(\O)\), solving the following convex saddle-point problem:
  \begin{align}
    p_{\text{loc}} &:= \sup_{\kappa \in \kadm} \check{\imath}_{\text{loc}}(\kappa),  &\qquad& \text{where} \label{eq:local_saddle} \\
    \check{\imath}_{\text{loc}}(\kappa) &:= \inf_{u \in \Uz} \check{I}_{\text{loc}}(\kappa;u),  \label{eq:local_primal} \\
    \check{I}_{\text{loc}}(\kappa;u) &:= \frac{1}{p}\int_{\O} \kappa(x)|\nabla u(x)|^p\dx - \ell(u), &\qquad& \text{and} \label{eq:local_primal1}\\
    \ell(u) &:= \int_{\O} f(x)u(x)\dx. \label{eq:ell}
  \end{align}
For both physical and solvability reasons we will assume that there are constants \(0<\underline{\kappa}\leq \overline{\kappa}<+\infty\), such that
\begin{equation}\label{eq:kadm_loc}
    \underline{\kappa}\leq \kappa(x) \leq \overline{\kappa}, \qquad \text{a.e.\ in \(\O\)}, \forall \kappa \in \kadm,
\end{equation}
and that the admissible set of conductivities \(\kadm\) is nonempty, convex, and weakly\(^*\) compact in \(L^\infty(\O)\).
A model example of \(\kadm\) is given by
\[
\kadm = \bigg\{\, \kappa \in L^\infty(\O)
    \mid \kappa(x) \in [\underline{\kappa},\overline{\kappa}], \text{\ for almost all \(x\in\O\)},
    \int_{\O} \kappa(x)\dx \leq V
    \,\bigg\},
\]
where \(V>0\) is an upper bound on the total available conductivity.

We now outline two ideas, which provide the proper context for this work.

\subsection{Dual formulation of the state problem}
\label{subsec:dual}

It is often desirable to replace the ``inner'' optimization problem~\eqref{eq:local_primal} in the saddle point problem~\eqref{eq:local_saddle} with its convex conjugate.
In the present case, this leads to the following dual variational statement:
\begin{align}
    \hat{\imath}_{\text{loc}}(\kappa) &= \inf_{\sigma \in \Q(f)} \hat{I}_{\text{loc}}(\kappa;\sigma),
    &\qquad&\text{where} \label{eq:local_dual}\\
    \hat{I}_{\text{loc}}(\kappa;\sigma) &= \frac{1}{q}\int_{\O} \kappa^{1-q}(x)|\sigma(x)|^q\dx,\\
    \Q &= \{\, \tau \in L^q(\O;\Rn) \colon \Div\tau \in L^q(\O)\,\},
\end{align}
\(\Div\) is the distributional divergence operator, and \(\Q(f) = \{\, \tau \in \Q \colon \Div\tau =f \,\}\) is an affine subspace of \(\Q\).
Furthermore, \(\Q\) is a separable, reflexive Banach space with respect to the graph norm \(\|\sigma\|_{\Q}^q = \|\sigma\|_{L^q(\O;\Rn)}^q + \|\Div \sigma\|_{L^q(\O)}^q\),
which also makes \(\Q(f)\) into a closed subset of \(\Q\).

As we shall recall, the problem~\eqref{eq:local_dual} is solvable for each \(f\in L^q(\O)\), and the equality \(\check{\imath}_{\text{loc}}(\kappa) + \hat{\imath}_{\text{loc}}(\kappa) = 0\) holds.
This allows us to equivalently restate the problem~\eqref{eq:local_saddle} as
\begin{equation}\label{eq:local_min}
    d_{\text{loc}} := \inf_{\kappa \in \kadm} \hat{\imath}_{\text{loc}}(\kappa). 
\end{equation}

The numerous reasons for considering~\eqref{eq:local_min} in lieu of~\eqref{eq:local_saddle} may now become clearer.
Physically, the conjugate variables \(\sigma\) provide us with valuable information about the propagation of the quantity \(u\) through the heterogeneous domain \(\O\);  for example in the case of heat conduction they correspond to heat fluxes.
Algorithmically, one may interchange the order of minimization in~\eqref{eq:local_min} and~\eqref{eq:local_dual}, or indeed minimize simultaneously with respect to the conduction coefficients and the fluxes, see for example~\cite{evgrafov2015chebyshev,kocvara2016,kocvara2020,ioannis:2021}.
Analytically, one derives useful optimality conditions with respect to \(\kappa\) from the assumed knowledge of optimal fluxes \(\sigma\), see for example~\cite{cherkaev2012variational,allaire2012shape,bendsoe2013topology,evgrafov2021nonlocal} and the references therein.

\subsection{Nonlocal physics}
\label{subsec:nlphys}

Another avenue, along which~\eqref{eq:local_saddle} may be developed, is to assume that the diffusion is governed by a \emph{nonlocal} operator instead of a differential one.  Nonlocal models have become one of the active research directions lately owing to their numerous attractive mathematical properties and applications, see for example~\cite{SiLe10,BuVal,AnMaRoTo10,du,madenci_oterkus}.
In particular, nonlocal models arise naturally when long-range interactions between points in space and/or in time are considered. Nonlocality appears in a variety of contexts, such as image processing~\cite{GiOs08}, pattern formation~\cite{Fife03}, population dispersal~\cite{CoCoElMa07}, nonlocal diffusion~\cite{AnMaRoTo10,BuVal}, nonlocal characterization of Sobolev spaces~\cite{bourgain2001another,ponce2004estimate,ponce2004new} and fractional Laplacian \cite{fractional_laplacian} and applications of these.  In the current work we will focus on peridynamics, a nonlocal paradigm in solid mechanics, see for example~\cite{SiLe10,madenci_oterkus}.

We will utilize a model which corresponds to the bond-based peridynamic model of diffusion.  The ``degree of nonlocality'' of the model will be characterized by a small parameter  \(\delta>0\), which will be referred to as the nonlocal interaction horizon. In this model only the points, which are less than \(\delta\) distance units apart, may interact with each other.  To this end we let \(\Od = \cup_{x\in \O} B(x,\delta)\) be the domain \(\Omega\) including the nonlocal ``\(\delta\)-halo''/``collar''/boundary \(\Gamma_{\delta} = \Od\setminus\O\), where \(B(x,\delta) = \{\, z \in \Rn : |z-x|<\delta\,\}\).  The strength of the nonlocal interactions will be characterized by a radial kernel \(\Wd: \Rn\to \R_+\), which is supported in \(B(0,\delta)\) and is normalized to satisfy
\begin{equation}\label{eq:normalization}
    \int_{\Rn} |x|^p\Wd^p(x)\dx = K_{p,n}^{-1},\qquad \text{where}\qquad
    K_{p,n} = |\Sn|^{-1}\int_{\Sn} |e\cdot s|^p \,\mathrm{d}s,
\end{equation}
\(\Sn\) is the unit sphere in \(\Rn\), and \(e\in \Sn\) is arbitrary, see~\cite{bourgain2001another}.
Owing to the radial nature of \(\Wd\), sometimes we will write \(\Wd(|x|)\) instead of \(\Wd(x)\), in particular in connection with integration in spherical coordinates.
A typical example of \(\Wd\) is given by 
\[\Wd(x) = \begin{cases}c_{\delta,\alpha}|x|^{\alpha}, &\qquad |x|\leq \delta,\\
  0, &\qquad \text{otherwise},
\end{cases}
\] where \(\alpha \in (-1-n/p,-1]\) is a modelling parameter, and the normalization constant \(c_{\delta,\alpha}>0\) is computed from the normalization condition above.
More singular behaviour of \(\Wd\)  at zero (that is, more negative values of \(\alpha\)) results in higher solution regularity to nonlocal diffusion problems.

The linear operator \(\grad : D(\grad)\subseteq L^p(\O) \to L^p(\Od\times\Od)\) defined by
\begin{equation}\label{eq:nlgrad}\grad u(x,x') = [u(x)-u(x')]\Wd(x-x'), \quad (x,x')\in \Rn\times\Rn,\end{equation}
will be called the nonlocal gradient, where we employ the convention that functions are extended to vanish outside of their explicit domain of definition, unless explicitly stated otherwise.
We will recall and summarize its properties in Proposition~\ref{prop:grad0}.
For now, it is sufficient to state that 
\[
  \Udz = D(\grad) = \{\, u \in L^p(\O) \colon u|_{\R^n \setminus \O} \equiv 0,
  \grad u \in L^p(\Od\times\Od) \,\},
\]
is a separable, reflexive Banach space when equipped with the norm \(\|u\|_{\Udz} = \|\grad u\|_{L^p(\Od\times\Od)}\), which is furthermore equivalent (uniformly for small \(\delta>0\)) to the graph norm associated with the operator \(\grad\), i.e., \([\|u\|^p_{L^p(\O)}+\|\grad u\|^p_{L^p(\Od\times\Od)}]^{1/p}\).
The central result on which this construction hinges is the nonlocal Poincar{\'e} inequality, see for example~\cite{bourgain2001another,ponce2004estimate,ponce2004new}.

With this in mind we can now state the direct nonlocal analogue of~\eqref{eq:local_saddle}--\eqref{eq:local_primal1}:
\begin{align}
  p_{\delta} &:= \sup_{\ddot{\kappa} \in \kadm_{\delta}} \check{\imath}_{\delta}(\ddot{\kappa}),  &\qquad& \text{where} \label{eq:nl_saddle} \\
  \check{\imath}_{\delta}(\ddot{\kappa}) &:= \inf_{u \in \Udz} \check{I}_{\delta}(\ddot{\kappa};u),  \label{eq:nl_primal} \\
  \check{I}_{\delta}(\ddot{\kappa};u) &:= \frac{1}{p}\int_{\Od}\int_{\Od} \ddot{\kappa}(x,x')|\grad u(x,x')|^p\dx\dx' - \ell(u). \label{eq:nl_primal1}
\end{align}
The coefficient \(\ddot{\kappa}\) may be thought of as the ``nonlocal conductivity,'' in the sense that \(\ddot{\kappa}(x,x')\) provides an \emph{additional} characterization of the bond strength between the points \((x,x')\in \Od\times\Od\).  We emphasize the word additional, because in this model the bond strength depends not only on the distance between the points, which is accounted for in the radial kernel \(\Wd\), but also on what can be thought of as the material property \(\ddot{\kappa}(x,x')\).
The quantity \(\check{I}_{\delta}(\ddot{\kappa};u)\) defined in~\eqref{eq:nl_primal1} is the direct analogue of the usual, local, Dirichlet/potential energy given by~\eqref{eq:local_primal1} associated with \(p\)-Laplacian.
For now the exact structure of the set of admissible nonlocal conductivities  \(\kadmd\) is not going to be of great importance.  Similarly to~\eqref{eq:kadm_loc}, we will assume that nonlocal conductivities satisfy the bounds
\begin{equation}\label{eq:kadm_nlloc}
    \underline{\kappa}\leq \ddot{\kappa}(x,x') \leq \overline{\kappa}, \qquad \text{a.e.\ in \(\Od\times \Od\)}, \forall \ddot{\kappa} \in \kadm_{\delta},
\end{equation}
for some constants  \(0<\underline{\kappa}\leq \overline{\kappa}<+\infty\).
Similarly, we will assume that  \(\kadm_{\delta}\) is a nonempty, convex, and weakly\(^*\) compact set in \(L^\infty(\Od\times\Od)\).
Finally, we will assume that all nonlocal conductivities are \emph{symmetric}, that is,
\begin{equation}\label{eq:kappa_symm}
  \ddot{\kappa}(x,x') = \ddot{\kappa}(x',x), \qquad\text{for almost all \((x,x')\in \Od\times\Od\).}
\end{equation}
The symmetry assumption on \(\ddot{\kappa}(x,x')\) is reasonable because the quantity \(|\grad u(x,x')|^p\) satisfies it, and consequently the Dirichlet energy \(\check{I}_{\delta}(\ddot{\kappa};u)\) depends only on the symmetric part of \(\ddot{\kappa}(x,x')\).

In the context of nonlocal \(p\)-Laplacian the well-posedness of the state problem has been analyzed in for example in~\cite{hindsraduplap} utilizing the nonlocal vector calculus of \cite{gunzburgercalculus}, see also~\cite{bellidoexistence}. 
The corresponding saddle point problem~\eqref{eq:nl_saddle} has been introduced and studied in~\cite{andres2015nonlocal,andres2017convergence,evgrafov2019non,AnMuRo21,andres2023minimization}, even if in certain cases only for \(p=2\).

The especially interesting research question for these models is whether they can approximate the local models, in some apporiate sense,\footnote{For optimization problems one typically utilizes the language of \(\Gamma\)-convergence to formulate and answer such questions.} when one considers the limit \(\delta\to 0\) in for example~\eqref{eq:nl_primal} or~\eqref{eq:nl_saddle}.
These questions have been  answered affirmatively in the cited references,\footnote{We note that in order to properly formulate the  question of convergence for~\eqref{eq:nl_primal}, we need to make further assumptions about the relationship between the admissible nonlocal and local conductivities, i.e., \(\kadmd\) and \(\kadm\).
The reasonable assumption is that an admissible \(\ddot{\kappa}(x,x')\) is an average, typically, but not necessarily arithmetic, of admissible \(\kappa(x)\) and \(\kappa(x')\).  How one defines \(\kappa|_{\Od\setminus\O}\) is not very important.} with the critical tools provided in~\cite{bourgain2001another,ponce2004new,ponce2004estimate} and their generalization to problems with heterogeneous coefficients~\cite{munoz2021generalized,JulioMunoz2021CtGP}.

\subsection{The contribution of this work}

The main contribution of this work effectively lies in combining the lines of thought outlined in Subsections~\ref{subsec:dual} and~\ref{subsec:nlphys}.
More specifically, we will define the dual variational principle corresponding to the nonlocal state problem~\eqref{eq:nl_primal}.  We will demonstrate its well-posedness, strong duality with the primal statement~\eqref{eq:nl_primal}, and finally carry out the analysis of convergence towards the local problems as \(\delta \to 0\).
This plan has been successfully executed in the quadratic case \(p=2\) in~\cite{evgrafov2021dual}.
Its generalization to the case \(p \in (1,\infty)\)  is not quite straightforward. In particular, in the non-quadratic case we may no longer rely upon the global stability of solutions to the mixed linear variational principles (uniformly with respect to small \(\delta>0\)).
Furthermore, the explicit construction of a ``recovery sequence'' involved in obtaining a lower bound for the local problem (see Subsection~\ref{subsec:gammaconv_limsup} and Section~\ref{sec:Fconsist}) relies upon a nonlinear operator in the present case, which is also interesting.

Owing to the strong duality, the \(\Gamma\)-convergence result for the dual nonlocal problem is nearly identical to the one mentioned in Section~\ref{subsec:nlphys}, see for example~\cite{munoz2021generalized,JulioMunoz2021CtGP,AnMuRo21,andres2023minimization}\footnote{The only difference is that we assume arithmetic averaging of resistivities \(\kappa^{1-q}(x)\) and \(\kappa^{1-q}(x')\) when computing \(\ddot{\kappa}^{1-q}(x,x')\) and not the arithmetic averaging of conductivities; see the previous footnote.}.  However, our approach utilizes different ideas, namely the stability of optimal values to convex constrained optimization problems, and is of interest on its own merits.

\subsection{Outline}

The outline of the rest of this paper is as follows.
In section~\ref{sec:nldiver} we introduce the nonlocal divergence operator, state the dual variational statement for the nonlocal \(p\)-Laplacian, and establish its well-posedness and quantitative stability.
With these results in place, in Section~\ref{sec:nlcontrol} we state the nonlocal analogue of~\eqref{eq:local_dual}, which is the main subject of study in the present work, and establish its well-posedness.
Finally, Sections~\ref{sec:localization} and~\ref{sec:Fconsist} deal with the question of \(\Gamma\)-convergence of the nonlocal control in the coefficients towards its local counterpart.
We have chosen to collect most of the computations needed to establish the consistency of the limit of the recovery sequence constructed in Subsection~\ref{subsec:gammaconv_limsup} and needed for \(\Gamma\)-convergence into Section~\ref{sec:Fconsist}.

\section{Nonlocal divergence and the nonlocal dual variational principle for \(p\)-Laplacian}
\label{sec:nldiver}

Let us begin this section by summarizing the well-known but pertinent properties of the nonlocal gradient operator defined in~\eqref{eq:nlgrad} for future reference.
\begin{proposition}\label{prop:grad0}
  The following statements hold.
  \begin{enumerate}
    \item \label{grad0:dense} The domain \(D(\grad)=\Udz\) is dense in \(L^p(\Omega)\).
    \item \label{grad0:clgr} The graph of \(\grad\) is closed in \(L^p(\Omega)\times L^p(\Od\times\Od)\).
    \item \label{grad0:poincare} There exist constants $\overline{\delta}>0$ and $C>0$ such that \(\|u\|_{L^p(\Omega)} \leq C \|\grad u\|_{L^p(\Od \times \Od)}\) for all $u \in \Udz$ and $\delta \in (0,\overline{\delta})$.
    In particular, \(\grad\) is injective.
    \item The range of \(\grad\), \(R(\grad)\), is closed in \(L^p(\Od \times \Od)\).
    \item \label{grad0:banach} \(\Udz\) with the norm \(\|u\|_{\Udz} = \|\grad u\|_{L^p(\Od \times \Od)}\) is a separable, reflexive Banach space.
  \end{enumerate}
\end{proposition}
\begin{proof}
  \begin{enumerate}
    \item We observe that \(\forall \phi \in C^1_c(\O)\) we have
    \(|\grad \phi(x,x')| \leq \|\nabla \phi\|_{C(\text{cl}\O;\Rn)} |x-x'|\Wd(x-x')\).
    In view of~\eqref{eq:normalization} we can conclude that  \(L^p(\O) = \text{cl} C^1_c(\O) \subseteq \text{cl} D(\grad) \subseteq L^p(\O)\).
    \item  Convergence in \(L^p\) implies a.e.\ pointwise convergence, up to a subsequence; \(\grad\) is defined pointwise.
    \item   The proof of the fact that the constant in the nonlocal Poincar{\'e}-type inequality may be chosen uniformly for all small \(\delta>0\) is found in for example~\cite[Lemma~4.2]{evgrafov2019non} for \(p=2\). The proof holds verbatim for \(1<p<\infty\).
    \item This is shown in~\cite[Theorem~2.21]{brezis2010functional} as a consequence of the previous points.
    \item We follow the standard argument and consider \(\Udz=D(\grad)\) equipped with the graph norm \(\|u\|^p=
    \|u\|_{L^p(\O)}^p + \|\grad u\|_{L^p(\Od\times\Od)}^p\), and an isometry
    \(T: \Udz \to L^p(\O) \times L^p(\Od\times\Od)\) given by \(Tu=(u,\grad u)\).
    The point~\ref{grad0:clgr} implies that \(R(T)=\graph \grad\) is a closed subspace of the the reflexive and separable Banach space \(L^p(\O) \times L^p(\Od\times\Od)\), and is consequently also reflexive, separable, and complete~\cite[Propositions~3.20, 3.25]{brezis2010functional}.
    The same properties hold for \(\Udz\) equipped with the graph norm because \(T\) is a linear isometry.
    Finally, the conclusion does not change when we equip \(\Udz\) with the norm \(\|u\|_{\Udz} = \|\grad u\|_{L^p(\Od \times \Od)}\), since the this norms is equivalent to the graph norm in view of~\ref{grad0:poincare}.\qedhere
  \end{enumerate}
\end{proof}

Owing to the density of \(D(\grad)\) in \(L^p(\O)\) we can define a negative adjoint operator \(\diver: D(\diver) \subset L^q(\Od\times\Od) \to L^q(\O)\), where we use Riesz representation (see, for example, \cite[Theorem~4.11]{brezis2010functional}) to identify the duals of \(L^p\)-spaces with \(L^q\):
\begin{equation}\label{eq:intbyparts}
  \int_{\O} \diver \ddot{\sigma}(x) u(x)\dx
  =
  -\int_{\Od}\int_{\Od} \ddot{\sigma}(x,x') \grad u(x,x')\dx\dx',
  \qquad \forall (\ddot{\sigma},u) \in D(\diver)\times D(\grad).
\end{equation}
We will refer to this operator as the nonlocal divergence operator.
As an adjoint operator, it is closed and linear.
We summarize some of its properties we are going to utilize below.
\begin{proposition}\label{prop:diver}
  The following statements hold.
  \begin{enumerate}
      \item \label{diver:formula} The domain of \(\diver\), \(\Qd = D(\diver)\), is dense in \(L^q(\Od \times \Od)\). In fact, for each \(\qnl \in C^{0,1}(\Rn\times\Rn)\) we have the formula
      \[
        \diver\qnl(x)=\int_{\Od}[\qnl(x',x)-\qnl(x,x')]\Wd(x-x')\dx'.
      \]
      \item \(\Qd\) equipped with graph norm \(\|\qnl\|^q_{\Qd} = \| \qnl \|^q_{L^q(\Od \times \Od)} + \| \diver \qnl \|^q_{L^q(\Omega)}\) is a separable, reflexive Banach space.
      \item \label{diver:surj} \(\diver:\Qd \to L^q(\Omega)\) is a bounded, surjective linear operator.
  \end{enumerate}
\end{proposition}
\begin{proof}
   \begin{enumerate}
    \item
    The claim follows from the definition of \(\diver\) using the integration by parts formula.
    Namely, for all \(\qnl\in C^{0,1}(\Od\times\Od)\),
    and for all \(v\in\Udz\) we have, owing to Fubini's theorem:
    \begin{equation*}
      \begin{aligned}
      -\int_{\Od}\int_{\Od} &\qnl(x,x')\grad v(x,x') \dx'\dx =
      -\int_{\Od}\int_{\Od} \qnl(x,x')[v(x)-v(x')]\Wd(x-x')\dx'\dx
      \\&=
      -\lim_{\epsilon\searrow 0} \iint_{O_\epsilon} \frac{\qnl(x,x')}{|x-x'|}[v(x)-v(x')]|x-x'|\Wd(x-x')\dx'\dx
      \\&=
      \lim_{\epsilon\searrow 0} \iint_{O_\epsilon} \bigg[\frac{\qnl(x,x')}{|x-x'|}v(x')|x-x'|\Wd(x-x')
      -
      \frac{\qnl(x,x')}{|x-x'|}v(x)|x-x'|\Wd(x-x')\bigg]\dx'\dx
      \\&=
       \lim_{\epsilon\searrow 0} \iint_{O_\epsilon}v(x)[\qnl(x',x)-\qnl(x,x')]\Wd(x-x')\dx'\dx
      \\&=\int_{\O} v(x)\underbrace{\bigg\{\int_{\Od}[\qnl(x',x)-\qnl(x,x')]\Wd(x-x')\dx'\bigg\}}_{=\diver \qnl(x)}\dx,
      \end{aligned}
    \end{equation*}
    where \(O_\epsilon = \{\, (x,x')\in\Od\times\Od \colon |x-x'|>\epsilon \,\}\).
    The last equality holds owing to the continuity of Lebesgue's integral and owing to the following estimate, which holds for each \(x\in\O\):
    \begin{equation*}
      \begin{aligned}
        \bigg|\int_{\Od}[&\qnl(x',x)-\qnl(x,x')]\Wd(x-x')\dx'\bigg| \leq
        \int_{\Od}\frac{|\qnl(x',x)-\qnl(x,x')|}{|x-x'|}|x-x'|\Wd(x-x')\dx'\\
        &\leq 
        \bigg\{\int_{\Od} \bigg[\underbrace{\frac{|\qnl(x',x)-\qnl(x,x')|}{|x-x'|}}_{\leq \sqrt{2}\|\qnl\|_{C^{0,1}(\Od\times\Od)}}\bigg]^{q}\dx'\bigg\}^{1/q}
        \bigg\{\int_{\Od} [|x-x'|\Wd(x-x')]^p \dx'\bigg\}^{1/p}
        \\&\leq \sqrt{2}\|\qnl\|_{C^{0,1}(\Od\times\Od)} |\Od|^{1/q} K_{p,n}^{-1/p},
      \end{aligned}
    \end{equation*}
    where we have utilized H{\"o}lder inequality and~\eqref{eq:normalization}.
    Consequently, \(\diver\qnl \in L^{\infty}(\O)\subset L^q(\O)\) in this case.
    \item We note that \(\diver\) is a closed operator as the adjoint operator. One can then follow the proof of Proposition~\ref{prop:grad0}, point~\ref{grad0:banach}.
    \item The surjectivity of \(\diver\) follows from Proposition~\ref{prop:grad0}, points~\ref{grad0:dense}, \ref{grad0:clgr}, and~\ref{grad0:poincare} and~\cite[Theorem~2.21]{brezis2010functional}.
    The fact that it is a bounded operator follows directly from the definition of the graph norm on \(\Qd\).\qedhere
   \end{enumerate}
\end{proof}

Proposition~\ref{prop:diver}, point~\ref{diver:surj} implies that the affine subspace
\(\Qd(f) = \{\,\qnl \in \Qd \,|\, \diver \qnl = f\,\}\) is a nonempty and closed subset of \(\Qd\) for each \(f\in L^q(\Omega)\).
This allows us to formulate the following nonlocal version of~\eqref{eq:local_dual}, that is, the dual variational principle for \(p\)-Laplacian:
\begin{align} \label{eq:nonlocal_dual}
  \hat{\imath}_{\delta}(\knl) &= \inf_{\qnl \in \Qd(f)} \hat{I}_{\delta}(\knl;\qnl),
  &\qquad&\text{where}\\
  \hat{I}_{\delta}(\knl;\qnl) &= \frac{1}{q}\int_{\Od}\int_{\Od} \knl^{1-q}(x,x')|\qnl(x,x')|^q\dx\dx'.
\end{align}

\begin{theorem}[Existence of solutions and optimality conditions]
  \label{thm:nonlocal_state_existence}
  The problem~\eqref{eq:nonlocal_dual} admits a unique optimal solution \(\qnl^*\in \Qd(f)\) for each small \(\delta>0\), \(f \in L^{q}(\O)\), and \(\knl \in \kadmd\).
  Furthermore, there is a unique Lagrange multiplier \(u^*\in L^p(\O)\) such that the pair \((\qnl^*,u^*)\) satisfies the first order necessary and sufficient optimality conditions:
  \begin{equation}\label{eq:kkt}\begin{aligned}
    \int_{\Od}\int_{\Od} \knl^{1-q}(x,x')|\qnl^*(x,x')|^{q-2}\qnl^*(x,x')\pnl(x,x')\dx\dx'
    - \int_{\O}\diver \pnl(x) u^*(x) &= 0, &\qquad &\forall \pnl \in \Qd,\\
    \diver \qnl^* &= f.
  \end{aligned}\end{equation}
\end{theorem}
\begin{proof}
  The infimum in~\eqref{eq:nonlocal_dual} is attained: indeed we minimize a convex, continuous, and coercive function over the convex, closed, and nonempty subset of a reflexive Banach space \(\Qd\), and consequently the generalized Weierstrass theorem is applicable~\cite[Theorem 7.3.7]{kurdila2006convex}.
  Thereby the existence of \(\qnl^*\) is established.

  The uniqueness of \(\qnl^*\) is implied by the strict convexity of the integrand, i.e.\ the map \(\R\ni\sigma \mapsto \knl^{1-q}(x,x')|\sigma|^q\), \(1<q<\infty\).

  We now note that Robinson's constraint qualification (cf.~\cite[Equation (3.12)]{bonnans2013perturbation}) holds:
  \begin{equation}\label{eq:rob}
    0 \in \text{int}[R(\diver)-f] = L^q(\O),
  \end{equation}
  owing to the surjectivity of \(\diver\), see Proposition~\ref{prop:diver}, point~\ref{diver:surj}.
  Therefore, we can apply \cite[Theorem~3.6]{bonnans2013perturbation} to assert the necessity of the optimality conditions~\eqref{eq:kkt} and the existence of Lagrange multipliers.
  The sufficiency of~\eqref{eq:kkt} for optimality under convexity assumptions is standard, see for example~\cite[Proposition~3.3]{bonnans2013perturbation}.

  The uniqueness of \(u^*\) may be verified directly from~\eqref{eq:kkt} using the surjectivity of \(\diver\), see Proposition~\ref{prop:diver}, point~\ref{diver:surj}.
  Indeed, the difference \(v \in L^p(\O)\) between any two Lagrange multipliers would have to satisfy the equality
  \[\int_{\O} \diver\pnl(x)v(x)\dx = 0, \qquad \forall \pnl \in \Qd,\]
  which is easily obtainable by subtracting two versions of optimality conditions~\eqref{eq:kkt} corresponding to two potential Lagrange multipliers from each other.\footnote{%
  A less direct way of arriving at the same conclusion is by observing that the strict Robinson's constraint qualification, see~\cite[Definition~4.46]{bonnans2013perturbation} reduces to~\eqref{eq:rob} in the present situation and then by appealing to~\cite[Proposition~4.47~(i)]{bonnans2013perturbation}.}
\end{proof}

Symmetry will play a critical role in some of the forthcoming results.
In order to describe it we introduce the following notation for two subspaces of
symmetric and anti-symmetric functions in \(L^q(\Od\times\Od)\):
\[\begin{aligned}
  L^q_s(\Od\times\Od) &= \{\, \qnl \in L^q(\Od\times\Od)
\mid \qnl(x,x')-\qnl(x',x)=0, \text{a.e.\ in \(\Od\times\Od\)}\,\}\quad\text{and}\\
L^q_a(\Od\times\Od) &= \{\, \qnl \in L^q(\Od\times\Od)
\mid \qnl(x,x')+\qnl(x',x)=0, \text{a.e.\ in \(\Od\times\Od\)}\,\}.\\
\end{aligned}\]
We will also write \(\Qd^a = \Qd\cap L^q_a(\Od\times\Od)\),
and \(\Qd^a(f) = \Qd(f)\cap L^q_a(\Od\times\Od)\).
To characterize the symmetric and antisymmetric functions we will use the following auxiliary result.
\begin{lemma}[Symmetric and antisymmetric functions]\label{lem:sym}
  The following statements hold.
  \begin{enumerate}
    \item
  Subspaces \(L^q_s(\Od\times\Od)\) and \(L^q_a(\Od\times\Od)\) are complementary in \(L^q(\Od\times\Od)\).
  \item The set \(L^q_a(\Od\times\Od) \cap C^\infty_c(\Rn\times\Rn)\) is dense in \(L^q_a(\Od\times\Od)\). A similar statement holds for \(L^q_s(\Od\times\Od)\).
  \item \(\forall \qnl \in L^p_a(\Od\times\Od)\) and \(\pnl \in L^q_s(\Od\times\Od)\) we have the equality \(\int_{\Od}\int_{\Od} \qnl(x,x')\pnl(x,x')\dx\dx' = 0\).
  \item \label{lem:sym:orth} Suppose that \(\qnl \in L^p(\Od\times\Od)\) and \(\int_{\Od}\int_{\Od} \qnl(x,x')\pnl(x,x')\dx\dx'=0\), \(\forall \pnl \in L^q_s(\Od\times\Od) \cap C^\infty_c(\Rn\times\Rn)\). Then \(\qnl \in L^p_a(\Od\times\Od)\).
  The corresponding statement obtained by replacing \(s\) and \(a\) also holds.
  \end{enumerate}
\end{lemma}
\begin{proof}
  \begin{enumerate}
    \item Both \(L^q_s(\Od\times\Od)\) and \(L^q_a(\Od\times\Od)\) are closed, since convergence in \(L^q(\Od\times\Od)\) implies the pointwise convergence along a subsequence.
    Furthermore, from their definition we have \(L^q_s(\Od\times\Od)\cap L^q_a(\Od\times\Od) = \{0\}\).
    Finally, \(L^q_s(\Od\times\Od) + L^q_a(\Od\times\Od) = L^q(\Od\times\Od)\), because 
    each \(\qnl \in L^q(\Od\times\Od)\) may be represented as a sum of a symmetric and an antisymmetric functions: 
    \[\qnl(x,x') = \underbrace{\tfrac{1}{2}[\qnl(x,x')+\qnl(x',x)]}_{\in L^q_s(\Od\times\Od)} + \underbrace{\tfrac{1}{2}[\qnl(x,x')-\qnl(x',x)]}_{\in L^q_a(\Od\times\Od)}.\]
    \item This can be shown using ``symmetric mollification.''
    Indeed, let us take an arbitrary function \(\qnl \in L^q_a(\Od\times\Od) \), and an arbitrary mollifier \(\phi\in C^\infty_c(\Rn)\), \(\phi\geq 0\), \(\int_{\Rn}\phi(x)\dx=1\) and construct a family of symmetric mollifiers \(\psi_\epsilon(x,x') = \epsilon^{-2n}\phi(\epsilon^{-1}x)\phi(\epsilon^{-1}x')\).
    Then \(L^q_a(\Od\times\Od) \cap C^\infty_c(\Rn\times\Rn) \ni \psi_\epsilon \star \qnl
    \to \qnl\) in \(L^q(\Od\times\Od)\) as \(\epsilon \to 0\), see~\cite[Theorem~4.22]{brezis2010functional}.
    \item We perform the direct computation:
    \[\begin{aligned}\int_{\Od}\int_{\Od} \qnl(x,x')\pnl(x,x')\dx\dx'
    &= -\int_{\Od}\int_{\Od} \qnl(x',x)\pnl(x',x)\dx\dx'
    = -\int_{\Od}\int_{\Od} \qnl(x',x)\pnl(x',x)\dx'\dx
    \\&= -\int_{\Od}\int_{\Od} \qnl(x,x')\pnl(x,x')\dx\dx',\end{aligned}\]
    where the first equality stems from  the definition of the symmetric and antisymmetric functions, in the second we change the order of the integration (Fubini theorem, \cite[Theorem~4.5]{brezis2010functional}), and in the third we simply rename the variables \(x\leftrightarrow x'\).
    \item Let \(\qnl = \qnl^a + \qnl^s\), \(\qnl^a \in L^p_a(\Od\times\Od)\) and
    \(\qnl^s \in L^p_s(\Od\times\Od)\).
    We put \(\pnl^s = \text{sign}(\qnl^s) \in L^q_s(\Od\times\Od)\),
    and let \(\pnl^s_\epsilon \in L^q_s(\Od\times\Od) \cap C^\infty_c(\Rn\times\Rn)\) be the ``symmetrically mollified'' function.
    Then
    \[
      0 = \int_{\Od}\int_{\Od} \qnl(x,x')\pnl^s_\epsilon(x,x')\dx\dx'
      = \int_{\Od}\int_{\Od} \qnl^s(x,x')\pnl^s_\epsilon(x,x')\dx\dx'
      \xrightarrow[\epsilon\to 0]{ } \int_{\Od}\int_{\Od} |\qnl^s(x,x')|\dx\dx',
    \]
    where we have utilized the previously established points.  Thus \(\qnl^s\equiv 0\) and \(\qnl = \qnl^a \in L^p_a(\Od\times\Od)\).\qedhere
  \end{enumerate}
\end{proof}

\begin{corollary}[Antisymmetry of optimal fluxes]\label{cor:antisym}
  Any optimal solution \(\qnl^* \in \Qd(f)\) to~\eqref{eq:nonlocal_dual} must be \emph{antisymmetric}, that is, \(\qnl^* \in L^q_a(\Od\times\Od)\).
\end{corollary}
\begin{proof}
  Consider an arbitrary \(\pnl \in C^{0,1}(\Od\times\Od) \cap L^q_s(\Od\times\Od)\).
  Then \(\pnl \in \ker \diver\) owing to Proposition~\ref{prop:diver}, point~\ref{diver:formula}.
  Utilizing it in~\eqref{eq:kkt} we can see that the function \(\knl^{1-q}(x,x')|\qnl^*(x,x')|^{q-2}\qnl^*(x,x') \in L^p_a(\Od\times\Od)\), owing to Lemma~\ref{lem:sym}, point~\ref{lem:sym:orth}. Since each \(\knl(x,x') \in \kadmd\) is symmetric, we arrive at the inescapable conclusion that \(\qnl^*\) has to be antisymmetric.
\end{proof}

We will now show, that \(\diver\) possesses a bounded right inverse operator, whose norm is uniformly bounded for small \(\delta>0\).
This will allow us to make uniform a priori estimates of optimal solutions to~\eqref{eq:nonlocal_dual}.
\begin{proposition}\label{prop:quotientestimate}
    Let the constants \(\overline{\delta}>0\) and \(C>0\) be as in Proposition~\ref{prop:grad0}, point~\ref{grad0:poincare}.
    Then for each \(f \in L^q(\Omega)\) and \(\delta \in (0,\overline{\delta})\) there exists a function  \(\qnldf \in \Qd(f)\) with \(\|\qnldf\|_{L^q(\Od\times\Od)} \leq C \|f\|_{L^q(\Omega)}\).
    Consequently, \(\|\qnldf\|^q_{\Qd} \leq (C^q + 1)\|f\|^q_{L^q(\Omega)}\).
\end{proposition}
\begin{proof}
  Let us consider the following infima:
  \begin{equation}\label{eq:2infima}
    i_{\delta,f} = \inf_{\qnl \in \Qd(f)} \|\qnl\|_{L^q(\Od\times\Od)}
    = \inf_{\qnl \in \Qd(f), \pnl \in \ker \diver} \|\qnl+\pnl\|_{L^q(\Od\times\Od)}.
  \end{equation}
  The infimum on the left is clearly attained: we minimize a convex, continuous, and coercive function over the convex, closed, and nonempty subset of a reflexive Banach space \(\Qd\), and consequently the generalized Weierstrass theorem is applicable~\cite[Theorem 7.3.7]{kurdila2006convex}.
  We let \(\qnldf \in \Qd(f)\) be such that \(\|\qnldf\|_{L^q(\Od\times\Od)}=i_{\delta,f}\).

  If \(i_{\delta,f} = 0\) then also \(\qnldf = 0\), and we do not need to proceed any further since the claimed inequality clearly holds.

  If \(i_{\delta,f} > 0\) then \(\qnldf \not \in \ker \diver\), since \(i_{\delta,f}\) is also the shortest \(L^q(\Od\times\Od)\) distance between \(\Qd(f)\) and \(\ker\diver\), see the second infimum in~\eqref{eq:2infima}.
  We now define a linear functional \(F:\ker\diver \oplus \text{span}(\qnldf) \to \R\) by:
  \begin{equation*}
    F(\pnl+\alpha \qnldf) = \alpha i_{\delta,f},
    \qquad \forall \pnl\in\ker\diver.
  \end{equation*}
  Note that for each \(\qnl\in\ker\diver\) and \(\alpha \in \R\) we have the inequality
  \[
    |F(\qnl+\alpha \qnldf)| = |\alpha| \inf_{\pnl \in \ker \diver} \|\pnl+\qnldf\|_{L^q(\Od\times\Od)} = \inf_{\pnl \in \ker \diver} \|\pnl+\alpha\qnldf\|_{L^q(\Od\times\Od)} \leq \|\qnl+\alpha \qnldf\|_{L^q(\Od\times\Od)},
  \]
  where the second equality is owing to the fact that \(\ker\) is a linear subspace.
  We now apply Hahn--Banach theorem (cf.~\cite[Corollary~1.2]{brezis2010functional}) to extend \(F\) to a functional on \(L^q(\Od\times\Od)\) with \(\|F\|_{[L^q(\Od\times\Od)]'} \leq 1\), where we use the same symbol \(F\) to denote the extension.
  Note that our definition implies that \(F(\qnl) = 0\), \(\forall \qnl \in \ker \diver\). Owing to the closed range theorem (cf.~\cite[Theorem~2.19]{brezis2010functional}), we have that \(F \in (\ker \diver)^\perp = R(\grad)\), and consequently there is \(u \in \Udz\) such that
  \[F(\qnl) = \int_{\Od}\int_{\Od} \grad u(x,x')\qnl(x,x')\dx\dx',\qquad \forall \qnl \in L^q(\Od\times\Od),\]
  where we use Riesz representation (cf.~\cite[Theorem~4.11]{brezis2010functional}).
  In particular,
  \[\begin{aligned}
    \|\qnldf\|_{L^q(\Od\times\Od)}&=
    i_{\delta,f} =
    F(\qnldf) = \int_{\Od}\int_{\Od} \grad u(x,x')\qnldf(x,x')\dx\dx'
    = -\int_{\O} u(x)\diver\qnldf(x)\dx
    \\ &= -\int_{\O} u(x)f(x)\dx
    \leq \|u\|_{L^p(\Omega)} \|f\|_{L^q(\Omega)}  
    \leq C  \|\grad u\|_{L^p(\Od \times \Od)} \|f\|_{L^q(\Omega)}
    \\&=C \|F\|_{[L^q(\Od\times\Od)]'} \|f\|_{L^q(\Omega)} \leq C\|f\|_{L^q(\Omega)},
  \end{aligned} 
  \]
  where we have utilized the definition of \(\qnldf\), \(F\), \(\diver\) through the integration by parts, H{\"o}lder's inequality, Riesz representation, and most importantly the Poincar{\'e}-type inequality for \(\grad\), see Proposition~\ref{prop:grad0}, point~\ref{grad0:poincare}.

  Finally, the proof is concluded by observing that
  \[
    \|\qnldf\|^q_{\Qd} = 
    \|\qnldf\|^q_{L^q(\Od\times\Od)} + \|\diver \qnldf\|^q_{L^q(\O)}
    \leq (C^q + 1)\|f\|^q_{L^q(\Omega)}.\qedhere
  \]
\end{proof}
\begin{corollary}[A priori estimates]\label{cor:apriori}
  Let the constants \(\overline{\delta}>0\) and \(C>0\) be as in Proposition~\ref{prop:quotientestimate}, and let \(\knl \in \kadmd\) be arbitrary.
  Then for each \(f\in L^q(\O)\) and each \(\delta \in (0,\overline{\delta})\) the unique optimal solution \(\qnl^*\in\Qd(f)\) to~\eqref{eq:nonlocal_dual} satisfies the a priori estimates
  \begin{equation}\label{eq:apriori}
    \begin{aligned}
     \|\qnl^*\|^q_{L^q(\Od\times\Od)} &\leq \left(\frac{\overline{\kappa}}{\underline{\kappa}}\right)^{q-1} C^q\| f \|^q_{L^q(\O)}, \qquad\text{and}\\
     \|\qnl^*\|^q_{\Qd} &\leq \bigg[\left(\frac{\overline{\kappa}}{\underline{\kappa}}\right)^{q-1}C^q + 1\bigg]\| f \|^q_{L^q(\O)}.
    \end{aligned}
  \end{equation}
\end{corollary}
\begin{proof}
  Let the flux \(\qnldf\in \Qd(f)\) be as in Proposition~\ref{prop:quotientestimate}.
  Since it is feasible in~\eqref{eq:nonlocal_dual}, we have the following string of inequalities:
  \begin{equation*}
    \|\qnl^*\|_{L^q(\Od \times \Od)}^q \leq
    q \overline{\kappa}^{q-1}\hat{I}_{\delta}(\knl;\qnl^*) \leq 
    q \overline{\kappa}^{q-1}\hat{I}_{\delta}(\knl;\qnldf) \leq
    \left(\frac{\overline{\kappa}}{\underline{\kappa}}\right)^{q-1} \|\qnldf\|_{L^q(\Od \times \Od)}^q
    \leq \left(\frac{\overline{\kappa}}{\underline{\kappa}}\right)^{q-1} C^q \|f\|^q_{L^q(\O)}.
  \end{equation*} 
  Finally,
  \begin{equation*}
    \|\qnl^*\|_{\Qd}^q = \|\qnl^*\|_{L^q(\Od \times \Od)}^q + \|f\|^q_{L^q(\O)}
    \leq \bigg[\left(\frac{\overline{\kappa}}{\underline{\kappa}}\right)^{q-1}C^q + 1\bigg]\| f \|^q_{L^q(\O)}.\qedhere
  \end{equation*} 
\end{proof}

\begin{corollary}[Local Lipschitz semicontinuity of the optimal value, uniformly for small \(\delta>0\)]\label{cor:Lip_stab}
  Let the constants \(\overline{\delta}>0\) and \(C>0\) be as in Proposition~\ref{prop:quotientestimate}, and let \(\knl \in \kadmd\) be arbitrary.
  We will denote by \(\qnl_1\in \Qd(f_1)\) and \(\qnl_2\in \Qd(f_2)\) the unique optimal solutions to~\eqref{eq:nonlocal_dual} corresponding to the heat sources \(f=f_1\in L^q(\O)\) and \(f=f_2\in L^q(\O)\) respectively.
  Then there is a positive constant \(L=L(\|f_1\|_{L^q(\O)},q, C,\underline{\kappa})\), continuous and nondecreasing with respect to \(\|f_1\|_{L^q(\O)}\), but independent from \(\delta>0\) such that
  \[
    \Id_{\delta}(\knl;\qnl_2) \leq \Id_{\delta}(\knl;\qnl_1) + L\|f_2-f_1\|_{L^q(\O)}, \qquad \forall f_2 \in B(f_1,1).
  \]
\end{corollary}
\begin{proof}
  Instead of trying to fit the current situation into the framework of~\cite[Chapter 4]{bonnans2013perturbation}, we provide a direct and simple proof of this claim.
  Note that \(\qnl_1 + \qnl_{\delta,f_2-f_1} \in \Qd(f_2)\), 
  where we use the notation \(\qnl_{\delta,f}\) from Proposition~\ref{prop:quotientestimate}.
  Consequently
  \(\Id_{\delta}(\knl;\qnl_2) \leq \Id_{\delta}(\knl;\qnl_1+\qnl_{\delta,f_2-f_1}).\)
  In particular, utilizing the triangle inequality we get the estimate
  \[
    [\Id_{\delta}(\knl;\qnl_2)]^{1/q} \leq [\Id_{\delta}(\knl;\qnl_1)]^{1/q} + [\Id_{\delta}(\knl;\qnl_{\delta,f_2-f_1})]^{1/q} = \alpha+\beta.
  \]
  We apply the following estimates to the each of the terms:
  \[
    \begin{aligned}
      \alpha &\leq [\hat{I}_{\delta}(\knl;\qnl_{\delta,f_1})]^{1/q}
      \leq \left[\frac{1}{\underline{\kappa}^{q-1}q}\right]^{1/q}
      \|\qnl_{\delta,f_1}\|_{L^q(\Od \times \Od)}
      \leq \left[\frac{1}{\underline{\kappa}^{q-1}q}\right]^{1/q} C \|f_1\|_{L^q(\O)},\\
      \beta &= [\hat{I}_{\delta}(\knl;\qnl_{\delta,f_2-f_1})]^{1/q}
      \leq \left[\frac{1}{\underline{\kappa}^{q-1}q}\right]^{1/q}
      \|\qnl_{\delta,f_2-f_1}\|_{L^q(\Od \times \Od)}
      \leq \left[\frac{1}{\underline{\kappa}^{q-1}q}\right]^{1/q} C \|f_2-f_1\|_{L^q(\O)}.    \end{aligned}
  \]
  We can now apply Taylor's formula to the function \(\R_+\ni t\mapsto t^q \in \R_+\) to get the estimate
  \[\begin{aligned}
    \Id_{\delta}(\knl;\qnl_2)
    \leq
    [\alpha+\beta]^q
    = \alpha^q + q[\alpha+\theta_{\alpha,\beta}\beta]^{q-1}\beta
    \leq
    \alpha^q
    + q[\alpha+\beta]^{q-1}\beta
    \leq
    \Id_{\delta}(\knl;\qnl_1) + L\|f_2-f_1\|_{L^q(\O)},
  \end{aligned}\]
  where \(\theta_{\alpha,\beta} \in [0,1]\), we have utilized the non-negativity of
  \(\alpha\) and \(\beta\), and the monotonicity of the function \(\mathbb{R}_+ \ni t \mapsto t^{q-1}\in \R_+\).
  Finally, the Lipschitz constant in the estimate above can be explicitly expressed as 
  \(L = q[\alpha+\beta]^{q-1} \left[\frac{1}{\underline{\kappa}^{q-1}q}\right]^{1/q} C\).
  It remains to utilize the upper estimates on \(\alpha\) and \(\beta\), while having the assumption \(\|f_2-f_1\|_{L^q(\O)}\leq 1\) in mind, to conclude the proof.
\end{proof}

In order to rigorously recover the duality relationship between~\eqref{eq:nonlocal_dual} and~\eqref{eq:nl_primal} we are going to need the a slightly different characterization of \(\Udz\), which is akin to~\cite[Proposition~9.18]{brezis2010functional}, see also~\cite[Proposition 4.5]{evgrafov2021dual}.
\begin{proposition}\label{prop:Udz_char}
  Assume that \(u\in L^p(\O)\).  Then \(u\in \Udz\) if and only if there is a constant \(c\geq 0\) such that \(\forall \qnl \in \Qd^a \cap C^{0,1}(\Od\times\Od)\) we have the inequality
  \begin{equation}\label{eq:diver_adj0}
    \bigg|\int_{\O} u(x)\diver \qnl(x)\dx\bigg| \leq c \|\qnl\|_{L^q(\Od\times\Od)}.
  \end{equation}
\end{proposition}
\begin{proof}
  The ``only if'' part follows immediately from the integration by parts formula~\eqref{eq:intbyparts} and H{\"older} inequality; in this case we can use \(c=\|\grad u\|_{L^p(\Od\times\Od)}\).

  To obtain the ``if'' part, we first note that the assumed inequality holds trivially also for the symmetric Lipschitz continuous fluxes owing to Proposition~\ref{prop:diver}, part~\ref{diver:formula}.
  Consequently, for each \(\qnl \in C^{0,1}(\Od\times\Od)\) we can write
  \[\bigg|\int_{\O} u(x)\diver \qnl(x)\dx\bigg|
  =\bigg|\int_{\O} u(x)\diver \qnl_a(x)\dx\bigg|\leq c\|\qnl_a\|_{L^q(\Od\times\Od)}
  \leq c\|\qnl\|_{L^q(\Od\times\Od)},\]
  where in the last inequality we estimate the norm \(\|\qnl_a\|_{L^q(\Od\times\Od)}\) of the antisymmetric part \(\qnl_a(x,x')=\tfrac{1}{2}[\qnl(x,x')-\qnl(x',x)]\) of \(\qnl\) using triangle inequality.\footnote{Less directly, one may also appeal to the complementarity of the symmetric and antisymmetric fluxes (cf.\ Proposition~\ref{lem:sym}) and utilize~\cite[Theorem~2.10]{brezis2010functional} dealing with sums of closed subspaces.}

  Proceeding with integration by parts as in the proof of Proposition~\ref{prop:diver}, part~\ref{diver:formula}, we define a bounded linear functional \(F_u: C^{0,1}(\Od\times\Od)\to\R\) by
  \begin{equation}\label{eq:Fu}
    F_u(\qnl) = \int_{\O} u(x)\diver\qnl(x)\dx = -\lim_{\epsilon\to 0}\iint_{O_\epsilon} \qnl(x,x')\grad u(x,x')\dx\dx',
  \end{equation}
  where as before \(O_\epsilon = \{\, (x,x')\in\Od\times\Od \colon |x-x'|>\epsilon \,\}\).
  We now apply Hahn--Banach theorem (cf.~\cite[Corollary~1.2]{brezis2010functional}) to extend \(F_u\) to a functional on \(L^q(\Od\times\Od)\).
  Owing to Riesz representation (see, for example, \cite[Theorem~4.11]{brezis2010functional}), there is \(\pnl \in L^p(\Od\times\Od)\) such that
  \(F_u(\qnl) = \int_{\Od}\int_{\Od} \pnl(x,x')\qnl(x,x')\dx\dx'\), for all \(\qnl\in L^q(\Od\times\Od)\).
  For each \(\epsilon >0\) the functions \(C^{0,1}(O_\epsilon)\) are dense in \(L^q(O_\epsilon)\), and~\eqref{eq:Fu} implies that \(\pnl\) agrees with \(-\grad u\), almost everywhere in \(O_\epsilon\).
  Thus \(\grad u \in L^p(\Od\times\Od)\), and the proof is concluded.
\end{proof}

\begin{proposition}[Strong duality and regularity of Lagrange multipliers]
  \label{prop:duality_nlocal}
  The dual problem for~\eqref{eq:nonlocal_dual} is (equivalent to) \eqref{eq:nl_primal}.
  The dual problem admits a unique solution \(u^* \in \Udz\) for each small \(\delta>0\),
  each \(\knl \in \kadmd\), and each \(f\in L^q(\O)\).
  The strong duality holds, that is,
  \[\check{\imath}_{\delta}(\ddot{\kappa})+\hat{\imath}_{\delta}(\ddot{\kappa})=0.\]
\end{proposition}
\begin{proof}
  We compute the dual function \(\tilde{\imath}_{\knl}: L^p(\O)\to \R\cup \{\pm \infty\}\) for~\eqref{eq:nonlocal_dual} as
  \[\begin{aligned}
    \tilde{\imath}_{\knl}(u) &= \inf_{\qnl\in\Qd} 
    \bigg[\hat{I}(\knl;\qnl)
    + \int_{\O}[f(x)-\diver\qnl(x)]u(x)\dx\bigg]\\
    &=\begin{cases}
      -\infty, &\quad u \not\in \Udz,\\
      \displaystyle\ell(u) + \inf_{\qnl\in\Qd} 
     \int_{\Od}\int_{\Od}\bigg[\frac{1}{q}\knl^{1-q}(x,x')|\qnl(x,x')|^q+\qnl(x,x')\grad u(x,x')\bigg]\dx\dx', &\quad u \in \Udz,
    \end{cases}
  \end{aligned}\]
  where we have utilized Proposition~\ref{prop:Udz_char} to arrive at the first case, and the integration by parts in the second.
  To compute the remaining infimum, we utilize the fact that \(\Qd\) is dense in \(L^q(\Od\times\Od)\), see Proposition~\ref{prop:diver}.
  Therefore, for \(u\in\Udz\) we have
  \[\begin{aligned}
    \tilde{\imath}_{\knl}(u) &= 
      \ell(u) + \int_{\Od}\int_{\Od}\inf_{\sigma\in\R}\bigg[\frac{1}{q}\knl^{1-q}(x,x')|\sigma|^q+\sigma\grad u(x,x')\bigg]\dx\dx',
  \end{aligned}\]
  where we have utilized the separability of the problem.
  The last infimum is attained at 
  \(\sigma = -\knl(x,x')|\grad u(x,x')|^{\frac{p}{q}}\text{sign}(\grad u(x,x'))\), which in turn leads to
  \[\begin{aligned}
    \tilde{\imath}_{\knl}(u) = \begin{cases}
      -\infty, &\quad u \not\in \Udz,\\
      \displaystyle\ell(u) -\frac{1}{p} \int_{\Od}\int_{\Od}\knl(x,x') |\grad u(x,x')|^p\dx\dx', &\quad u \in \Udz,
    \end{cases}
    = \begin{cases}
      -\infty, &\quad u \not\in \Udz,\\
      -\check{I}_{\delta}(\knl;u), &\quad u \in \Udz.\end{cases}
  \end{aligned}\]
  The dual problem may thus be stated as
  \[
    \sup_{u \in L^p(\O)}\tilde{\imath}_{\knl}(u) = 
    \sup_{u \in \Udz}[-\check{I}_{\delta}(\knl;u)]
    = -\check{\imath}_{\delta}(\ddot{\kappa}).
  \]
  In particular, the Lagrange multiplier \(u^*\in L^p(\O)\), whose existence and uniqueness has been established in Theorem~\ref{thm:nonlocal_state_existence}, is also the solution to the dual problem and is therefore in \(\Udz\); see~\cite[Section~3.1.1]{bonnans2013perturbation}. Consequently, the strong duality holds, leading to the equality \(\check{\imath}_{\delta}(\ddot{\kappa})+\hat{\imath}_{\delta}(\ddot{\kappa})=0\).
\end{proof}

\section{Control in the coefficients of the nonlocal \(p\)-Lapacian}
\label{sec:nlcontrol}
Having done all the preparatory work in Section~\ref{sec:nldiver}, we now turn our attention to the nonlocal version of~\eqref{eq:local_min}:
\begin{equation}\label{eq:nl_min}
  d_{\delta} := \inf_{\knl \in \kadmd} \hat{\imath}_{\delta}(\knl).
\end{equation}
Owing to the strong duality between~\eqref{eq:nonlocal_dual} and~\eqref{eq:nl_primal} established in Proposition~\ref{prop:duality_nlocal}, this problem is equivalent to~\eqref{eq:nl_saddle}.
One may therefore appeal to the existence of solutions to~\eqref{eq:nl_saddle} to assert the existence of solutions to~\eqref{eq:nl_min}.
A direct proof of existence, utilizng the convexity of the problem, is not much more difficult.
\begin{theorem}\label{thm:existence_nldesign}
  Problem~\eqref{eq:nl_min} admits an optimal solution \(\knl^* \in \kadmd\) for each small \(\delta >0\) and each \(f \in L^q(\O)\).
\end{theorem}
\begin{proof}
  Owing to Theorem~\ref{thm:nonlocal_state_existence}, the claim is equivalent to establishing that the infimum 
  \begin{equation}\label{eq:dblinf}
    \inf_{(\knl,\qnl)\in \kadmd\times\Qd(f)} \hat{I}_{\delta}(\knl;\qnl),
  \end{equation} is attained.
  The latter follows using the direct method of calculus of variations while utilizing the convexity of the objective and the feasible set.

  Indeed, we note  that owing to the bound \(\knl \leq \overline{\kappa}\), for all \(\knl \in \kadmd\), the lower-level sets of \(\hat{I}_{\delta}\) are bounded in \(L^q(\Od\times\Od)\), uniformly with respect to \(\knl\in\kadmd\), owing to the lower bound
  \(
    \hat{I}_{\delta}(\knl;\qnl) \geq (\overline{\kappa}^{1-q}/q)\|\qnl\|_{L^q(\Od\times\Od)}^q.
  \)
  When intersected with \(\Qd(f)\), these sets are therefore bounded in \(\Qd\).

  Let now \(\{\knl_k, \qnl_k\}_{k=1}^\infty\) be a minimizing sequence for~\eqref{eq:dblinf}; without loss of generality we will assume the sequence \(\{ \hat{I}_\delta(\knl_k;\qnl_k) \}_{k=1}^\infty\) is non-increasing.
  By our assumption,  \(\kadmd\) is weakly\(^*\) sequentially compact in \(L^\infty(\Od\times\Od)\), and consequently \(\{\knl_k\}_{k=1}^\infty\) contains a weakly\(^*\) convergent subsequence (which we will not relabel), whose limit \(\knl^*\) is in \(\kadmd\).
  Note that this convergence also happens weakly in \(L^s(\Od\times\Od)\), for any \(1<s<\infty\).
  Furthermore, we can extract a subsequence out of \(\{\qnl_k\}_{k=1}^\infty\), converging weakly in \(\Qd\) to a limit \(\qnl^*\in \Qd(f)\); switching to this subsequence is also not relabeled.
  
  We now apply Mazur's lemma (see for example~\cite[Corollary 3.8]{brezis2010functional}).
  Thus there is a function \(M:\N\to\N\) and a sequence of sets of nonnegative real numbers \(\{\, \lambda(m)_k\colon k=m,\dots,M(m)\,\}\), such that
  \(\sum_{k=m}^{M(m)}\lambda(m)_k=1\) and
  \(\lim_{m\to\infty}\|\sum_{k=m}^{M(m)}\lambda(m)_k (\knl_k,\qnl_k) - (\knl^*,\qnl^*)\|=0\)
  in \(L^s(\Od\times\Od)\times \Qd\).
  In particular, there is a further subsequence of convex combinations of \((\knl_k,\qnl_k)\), say \(m'\), which converges almost everywhere towards \( (\knl^*,\qnl^*)\) (see for example~\cite[Theorem 4.9]{brezis2010functional}).
  Since the integrand \([\underline{\kappa},\overline{\kappa}]\times \R \ni (\kappa,\sigma) \mapsto \kappa^{1-q}|\sigma|^q \in \R_{+}\) is convex and nonnegative, we can apply   Fatou's lemma (see for example~\cite[Lemma 4.1]{brezis2010functional}) to arrive at the desired conclusion:
  \[
  \begin{aligned}
  \hat{I}_\delta(\knl^*;\qnl^*)
  &\leq \liminf_{m'\to\infty} \hat{I}_\delta\left(\sum_{k=m'}^{M(m')}\lambda(m')_k \knl_k;
  \sum_{k=m'}^{M(m')}\lambda(m')_k \qnl_k\right)
  \\&\leq
  \liminf_{m'\to\infty}\left[\sum_{k=m'}^{M(m')}\lambda(m')_k
  \hat{I}_\delta(\knl_k;\qnl_k)\right]
  \leq
  \liminf_{m'\to\infty}
  \hat{I}_\delta(\knl_{m'};\qnl_{m'}),
  \end{aligned}
  \]
  where in the last inequality we have utilized the assumed monotonicity of the sequence \(\{ \hat{I}_\delta(\knl_k;\qnl_k) \}_{k=1}^\infty\).
  Thus we have shown that \((\knl^*,\qnl^*)\) is an optimal solution to~\eqref{eq:dblinf}.
\end{proof}

\section{Localization}\label{sec:localization}

We will now investigate the connection between the problem~\eqref{eq:nl_min} and the corresponding local problem~\eqref{eq:local_min}.
In order to do this properly, we have to make assumptions about the relationship between the local and nonlocal conductivities, that is, between \(\kadm\) and \(\kadmd\).
We will first extend all functions in \(\kadm\) by some fixed constant in the interval \([\underline{\kappa},\overline{\kappa}]\) in \(\Rn\setminus \O\), although we note that other approaches are clearly possible. Once this is done, we will assume that
\begin{equation}\label{eq:kadmd1}
  \kadmd = \left\{\, \knl \in L^\infty(\Od\times\Od) \colon \knl(x,x') = \left[\frac{\kappa^{1-q}(x) + \kappa^{1-q}(x')}{2}\right]^{1/(1-q)},
  \quad \forall (x,x') \in \Od\times\Od, \quad \kappa \in \kadm \,\right\}.
\end{equation}
Generally speaking, there is no reason to assume that the set \(\kadmd\) defined in~\eqref{eq:kadmd1} is convex, and consequently Theorem~\ref{thm:existence_nldesign} as stated may not be applicable in the present case.
However, since there is a one-to-one correspondence between \(\kadmd\) and \(\kadm\) in this case, we can write (with a slight abuse of notation)
\begin{equation}\label{eq:ihat1}
  \begin{aligned}
  \hat{I}_{\delta}(\kappa;\qnl)
  &= \frac{1}{q}\int_{\Od}\int_{\Od} \left[\frac{\kappa^{1-q}(x) + \kappa^{1-q}(x')}{2}\right]|\qnl(x,x')|^q\dx\dx', \quad \text{and similarly}\\
  \hat{\imath}_{\delta}(\kappa) &= \inf_{\qnl\in\Qd(f)} \hat{I}_{\delta}(\kappa;\qnl),
\end{aligned}
\end{equation}
instead of \(\hat{I}_{\delta}(\knl;\qnl)\) and \(\hat{\imath}_{\delta}(\knl)\), having in mind that \(\kappa\in\kadm\) and \(\knl\in\kadmd\) are related through~\eqref{eq:kadmd1}.
With this notational agreement we have the trivial equality
\[
  d_\delta 
  = \inf_{\knl \in \kadmd}\hat{\imath}_{\delta}(\knl)
  = \inf_{(\knl;\qnl)\in\kadmd\times\Qd(f)} \hat{I}_{\delta}(\knl;\qnl)
  = \inf_{(\kappa;\qnl)\in\kadm\times\Qd(f)} \hat{I}_{\delta}(\kappa;\qnl) 
  = \inf_{\kappa\in\kadm}\hat{\imath}_{\delta}(\kappa).
\]
The reasoning of Theorem~\ref{thm:existence_nldesign} applies to the third infimum from the left: indeed the objective is an integral of a convex nonnegative function, uniformly coercive with respect to the second argument, and \(\kadm\) is convex and weakly\(^*\) compact in \(L^\infty(\Omega)\).  
This infimum is therefore attained at some \((\kappa^*,\qnl^*)\in \kadm\times \Qd(f)\).
As a result, the rest of the infima are also attained at the corresponding \(\knl^*\in\kadmd\).

Before we proceed with the estimates, we would like to record the following simple observation, which will be utilized in what follows.
If in~\eqref{eq:ihat1} we have \(\qnl\in L^q_a(\Od\times\Od)\), then 
\[
  \hat{I}_{\delta}(\kappa;\qnl) = \frac{1}{q}\int_{\Od}\kappa^{1-q}(x)\int_{\Od}|\qnl(x,x')|^q\dx'\dx.
\]

\subsection{An upper estimate for the local problem}\label{subsec:gammaconv_lininf}

For each \(\qnl \in L^{q}(\Od\times\Od)\) let us define the following ``flux recovery'' operator
\begin{equation}\label{eq:Rdef}
  R_\delta \qnl(x) = \int_{\Od} (x-x')\qnl(x,x')\Wd(x-x')\dx'.
\end{equation}

\begin{proposition}\label{prop:claim0}
  \(R_\delta \in \mathcal{B}(L^{q}(\Od\times\Od),L^q(\Od;\Rn))\)
  and \(\Idloc(\kappa;R_\delta \qnl)\leq \hat{I}_\delta(\kappa;\qnl)\),
  \(\forall \qnl \in L^q_a(\Od\times\Od)\)\footnote{Also \(\forall \qnl \in L^q_s(\Od\times\Od)\)} and
  \(\forall \kappa \in \kadm\).
\end{proposition}
\begin{proof}
  We begin by applying H{\"o}lder inequality to obtain the estimate
  \[\begin{aligned}
  |R_\delta \qnl(x)|^q &= \left[\int_{\Od} (x-x')\qnl(x,x')\Wd(x-x')\dx'\right]^q
  \\&\leq \int_{\Od} |\qnl(x,x')|^q \dx' \left[\int_{\Od} |x-x'|^p\Wd^p(x-x')\dx'\right]^{q/p}
  \leq K_{p,n}^{-q/p}\int_{\Od} |\qnl(x,x')|^q \dx',
  \end{aligned}\]
  where the last inequality is owing to the normalization of the kernel \(\Wd(\cdot)\).
  It only remains to integrate both sides with respect to \(x\) to arrive at the first claim.

  Let us now take an arbitrary \(\psi \in L^p(\O_\delta;\Rn)\), \(\kappa \in \kadm\),
  and \(\qnl \in L^q_a(\Od\times\Od)\).
  We apply H{\"o}lder inequality as follows:
  \[\begin{aligned}
  &\int_{\O} R_\delta \qnl(x)\cdot \psi(x)\dx
  =
  \int_{\O}\int_{\Od} \kappa^{(1-q)/q}(x) \kappa^{1/p}(x)\psi(x)\cdot(x-x')\qnl(x,x')\Wd(x-x')\dx'\dx
  \\&\leq
  \left[\int_{\O}\int_{\Od} \kappa^{1-q}(x) |\qnl(x,x')|^q \dx'\dx\right]^{1/q}
  \left[\int_{\O}\int_{\Od} \kappa(x)|\psi(x)\cdot(x-x')\Wd(x-x')|^p \dx'\dx\right]^{1/p}
  \\&\leq
  \left[q\hat{I}_\delta(\kappa; \qnl)\right]^{1/q}
  \left[\int_{\O}\kappa(x)|\psi(x)|^p\dx\right]^{1/p},
  \end{aligned}\]
  where we have used the fact that for all \(x\in \Od\) we have the inequality
  \begin{equation}\label{eq:BBM0}
    \begin{aligned}
   \int_{\Od}[\psi(x)&\cdot(x-x')\Wd(x-x')]^p\dx'
   \le
    \int_{B(0,\delta)} \bigg|\psi(x)\cdot\frac{z}{|z|}\bigg|^p |z|^p\Wd^p(z)\,\mathrm{d}z
    \\
    &=\int_0^\delta r^{n-1} r^p w^p_\delta(r) \left(\int_{\Sn} \left| \psi(x)\cdot s\right|^p\,\mathrm{d}s \right)\, \mathrm{d}r
    \\&\le |\psi(x)|^p
    \underbrace{\int_{{\Rn}} |z|^p w^p_\delta(z) \,\mathrm{d}z}_{=K_{p,n}^{-1}} \underbrace{|\Sn|^{-1}\int_{\Sn} \left|e\cdot s\right|^p\,\mathrm{d}s}_{=K_{p,n}}
    = |\psi(x)|^p,
  \end{aligned}
  \end{equation}
 where \(e\in\Sn\) is an arbitrary unit vector.
 It remains to take \(\psi(x) = \kappa^{1-q}(x)|R_\delta\qnl(x)|^{q-2}R_\delta \qnl(x)\):
\[\begin{aligned}
  q\Idloc(\kappa; R_\delta \qnl)
  &= \int_{\O} \kappa^{1-q}(x)|R_\delta\qnl(x)|^{q-2} R_\delta \qnl(x)\cdot R_\delta \qnl(x)\dx
  \\&\leq
  \left[q\hat{I}_\delta(\kappa; \qnl)\right]^{1/q}
  \left[\int_{\O}\kappa^{1+p(1-q)}(x)|R_\delta \qnl(x)|^{p(q-1)}\dx\right]^{1/p}
  =
  \left[q\hat{I}_\delta(\kappa; \qnl)\right]^{1/q}
  \left[q\Idloc(\kappa; R_\delta \qnl)\right]^{1/p}.\qedhere
\end{aligned}\]
\end{proof}

\begin{proposition}\label{prop:claim1_new}
  Assume that the sequence \(\qnl_{\delta_k} \in \Q_{\delta_k}(f)\), \(\delta_k \in (0,\bar{\delta})\),
  \(\lim_{k\to\infty}\delta_k = 0\), is bounded, and
  \(R_{\delta_k}\qnl_{\delta_k} \rightharpoonup \hat{\sigma}\), weakly
  in \(L^q(\O;\Rn)\).
  Then \(\hat{\sigma} \in \Q(f)\).
\end{proposition}
\begin{proof}
  Let \(\psi \in C^\infty_c(\O)\) be arbitrary.
  Utilizing the definitions~\eqref{eq:ell} and~\eqref{eq:Rdef} and the second order Taylor theorem for \(\psi\),
  we obtain the estimate
   \begin{equation*}
    \begin{aligned}
      &\bigg| \int_{\O_{\delta_k}} \nabla\psi(x)\cdot R_{\delta_k}\qnl_{\delta_k}(x)\dx + \ell(\psi)\bigg|
      =
      \bigg| \int_{\O_{\delta_k}} \nabla\psi(x)\cdot R_{\delta_k}\qnl_{\delta_k}(x)\dx
      +  \int_{\O}\mathcal{D}_{\delta_k} \qnl_{\delta_k}(x) \psi(x)\dx \bigg|
      \\
      &=
      \bigg| \int_{\O_{\delta_k}} \nabla\psi(x)\cdot R_{\delta_k}\qnl_{\delta_k}(x)\dx
      -  \int_{\O_{\delta_k}}\int_{\O_{\delta_k}} \ddot{\mathcal{G}}_{\delta_k} \psi(x,x') \qnl_{\delta_k}(x,x')\dx'\dx \bigg|
      \\
      &\leq\int_{\O_{\delta_k}}\int_{\O_{\delta_k}} |\nabla\psi(x)\cdot(x-x')-\psi(x)+\psi(x')||\qnl_{\delta_k}(x,x')|\omega_{\delta_k}(x-x')\dx'\dx
    \\&\leq
    \frac{1}{2}\|\nabla^2\psi\|_{L^\infty(\R^n;\R^{n\times n})}
    \int_{\O_{\delta_k}}\int_{\O_{\delta_k}}|\qnl_{\delta_k}(x,x')||x-x'|^2\omega_{\delta_k}(x-x')\dx'\dx
    \\&\leq
    \frac{\delta_k}{2}\|\nabla^2\psi\|_{L^\infty(\R^n;\R^{n\times n})}
    K_{p,n}^{-1/p}\|\qnl_{\delta_k}\|_{L^{q}(\O_{\delta_k}\times\O_{\delta_k})},
    \end{aligned}
  \end{equation*}
  where we have used H{\"o}lder inequality and the fact that \(\Wd \equiv 0\) for \(|x-x'|>\delta_k\).
  By passing to the limit \(\delta_k\to 0\) we arrive at the equality
  \[\int_{\O} \nabla\psi(x)\cdot \hat{\sigma}(x)\dx = -\int_{\O} f(x)\psi(x)\dx,\]
  which concludes the proof.
\end{proof}

We hereby obtained the following one-sided estimate.
\begin{proposition}\label{prop:liminf}
  Let \(\kadm\ni \kappa_{\delta_k} \wto \kappa\in \kadm\), weakly\(^*\) in \(L^\infty(\O)\), as \(\delta_k\to 0\).
  Then \(\hat{\imath}_{\text{loc}}(\kappa) \leq \liminf_{k\to\infty} \hat{\imath}_{\delta_{k}}(\kappa_{\delta_k})\).
\end{proposition}
\begin{proof}
  Suppose that the claim is false.  Then there is \(\epsilon >0\) and a sequence of indices \(k'\) such that \(\hat{\imath}_{\delta_{k'}}(\kappa_{\delta_{k'}})\leq \hat{\imath}_{\text{loc}}(\kappa)-\epsilon\), for all \(k'\).

  Let \(\qnl_{\delta_k} = \argmin_{\qnl\in\Q_{\delta_k}(f)}\hat{I}_{\delta_k}(\ddot{\kappa}_{\delta_k},\qnl)\), cf.~Theorem~\ref{thm:nonlocal_state_existence} which establishes its existence and uniqueness.
  Note that \(\qnl_{\delta_k} \in L^q_a(\O_{\delta_k}\times\O_{\delta_k})\), owing to Corollary~\ref{cor:antisym}.
  Furthermore, the sequence \(\|\qnl_{\delta_k}\|_{\Q_{\delta_k}}\) is bounded, see the a priori estimate given in Proposition~\ref{prop:quotientestimate}.
  Consequently, the sequence \(R_{\delta_k}\qnl_{\delta_k}\) is bounded in \(L^q(\O;\Rn)\), see Proposition~\ref{prop:claim0}.
  Since \(L^q(\O;\Rn)\), \(1<q<\infty\) is reflexive, every bounded sequence contains a weakly convergent subsequence, see \cite[Theorem 3.18]{brezis2010functional}.

  We now extract a further subsequence of indices \(k''\) out of \(k'\), such that
  \(R_{\delta_{k''}}\qnl_{\delta_{k''}} \wto \hat{\sigma}\), as \(k''\to \infty\), weakly in \(L^q(\O;\Rn)\).
  According to Proposition~\ref{prop:claim1_new} we have that \(\hat{\sigma}\in\Q(f)\), and consequently
  \[
    \hat{\imath}_{\text{loc}}(\kappa) = \inf_{\sigma\in\Q(f)} \Idloc(\kappa; \sigma)
    \leq \Idloc(\kappa; \hat{\sigma}).
  \]
  We now proceed as in Theorem~\ref{thm:existence_nldesign}.
  We note that \(\kappa_{\delta_k}\wto\kappa\) also in \(L^s(\O)\), \(1<s<\infty\).
  We utilize Mazur's lemma (see for example~\cite[Corollary 3.8]{brezis2010functional}) and construct to a sequence of convex combinations of \((\kappa_{\delta_{k''}},R_{\delta_{k''}}\qnl_{\delta_{k''}})\) which converges strongly 
  towards \((\kappa, \hat{\sigma})\).
  We switch to a further subsequence (not relabeled), which converges pointwise, and utilize the convexity and the nonnegatity of the integrand \([\underline{\kappa},\overline{\kappa}]\times \R \ni (\kappa,\sigma) \mapsto \kappa^{1-q}|\sigma|^q \in \R_{+}\) defining \(\Idloc\) and Fatou's lemma (see for example~\cite[Lemma 4.1]{brezis2010functional}).
  All this leads to the inequality
  \[
    \Idloc(\kappa;\hat{\sigma}) \leq \liminf_{k''\to\infty}
    \sum_{m=k''}^{M(k'')}\lambda_m(k'')  \Idloc(\kappa_{\delta_{m}},R_{\delta_{m}}\qnl_{\delta_{m}})
  \]
  for some nonnegative real numbers \(\lambda_{k''}(k''), \dots, \lambda_{M(k'')}(k'')\) adding up to one.
  It remains to utilize Proposition~\ref{prop:claim0} to arrive at the contradiction:
  \[
    \sum_{m=k''}^{M(k'')}\lambda_m(k'')  \Idloc(\kappa_{\delta_{m}},R_{\delta_{m}}\qnl_{\delta_{m}})
    \leq 
    \sum_{m=k''}^{M(k'')}\lambda_m(k'') \hat{I}_{\delta_{m}}(\kappa_{\delta_{m}},\qnl_{\delta_{m}}) = \sum_{m=k''}^{M(k'')}\lambda_m(k'') \hat{\imath}_{\delta_{m}}(\kappa_{\delta_{m}}) \leq \hat{\imath}_{\text{loc}}(\kappa)-\epsilon.
  \qedhere\]
\end{proof}

\begin{corollary}\label{thm:nonlocal_approximation_1sided}
  The optimal values of the local~\eqref{eq:local_min} and nonlocal~\eqref{eq:nl_min} problems
  satisfy the inequality
  \(d_{\text{loc}} \leq \liminf_{\delta\to 0} d_{\delta}\).
\end{corollary}
\begin{proof}
  Consider the family of optimal solutions \(\kappa^*_{\delta}\in\kadm\) to~\eqref{eq:nl_min}.
  From any subsequence of \(\{\kappa^*_{\delta}\}\) we can extract a weakly\(^*\) convergent subsequence, since \(\kadm\) is weakly\(^*\) sequentially compact in \(L^\infty(\O)\).
  To each such subsequence we can apply Proposition~\ref{prop:liminf}.
\end{proof}

\subsection{A lower estimate for the local problem}\label{subsec:gammaconv_limsup}

Let us define the following nonlinear operator, which maps vector-valued functions \(\sigma: \Rn\to\Rn\) to two-point quantities \(\ddot{F}_{\delta}\sigma: \Rn\times\Rn\to\R\):
\begin{equation}\label{eq:Fddot}
  \begin{aligned}
  \ddot{F}_{\delta}\sigma(x,x') = 
  \frac{1}{2}\{
    &|\sigma(x)|^{2-p}|\sigma(x)\cdot (x-x')|^{p-2}[\sigma(x)\cdot (x-x')]
    \\+
    &|\sigma(x')|^{2-p}|\sigma(x')\cdot (x-x')|^{p-2}[\sigma(x')\cdot (x-x')]
  \} \Wd^{p-1}(x-x').
  \end{aligned}
\end{equation}
Note that in the quadratic case \(p=2\), \(\ddot{F}_{\delta}\) reduces to the Hilbert space adjoint operator for \(R_{\delta}\), when the latter is restricted to \(L^2_a(\Od\times\Od)\), see~\cite[Section~6]{evgrafov2021dual}.

\subsubsection{An informal derivation of \(\ddot{F}_\delta\)}\label{sec:deriv_ddotF}

In this subsection we present an informal derivation of~\eqref{eq:Fddot}.
Throughout the derivation we shall assume that the conductivity field \(\kappa(x)\equiv \kappa\) is constant, and consequently also \(\knl \equiv \kappa\).
Let \(\qnl \in L^q(\Od\times\Od)\) be a two-point flux corresponding to a sufficiently smooth temperature field, for example \(u \in C^2_c(\O)\).
We can then state the identity
\begin{equation}\label{eq:a1}
    \begin{aligned}
        \qnl(x,x') &= -\knl(x,x')|\grad u(x,x')|^{p-2}\grad u(x,x')\\
        &= -\frac{1}{2}\kappa|\grad u(x,x')|^{p-2}\grad u(x,x')
        +\frac{1}{2}\kappa|\grad u(x',x)|^{p-2}\grad u(x',x),
    \end{aligned}
\end{equation}
where the first equality is obtained by isolating \(\qnl\) from the first of the optimality conditions~\eqref{eq:kkt}, and the second one is owing to the symmetry of \(\knl\) and anti-symmetry of \(\grad\).
We can now make the following finite difference approximations stemming directly from the definition of \(\grad\), see~\eqref{eq:nlgrad}:
\begin{equation}\label{eq:a2}
    \begin{aligned}
        \grad u(x,x')&\approx \nabla u(x)\cdot (x-x')\Wd(x-x'),\qquad\text{and}\\
        \grad u(x',x)&\approx \nabla u(x')\cdot (x'-x)\Wd(x-x'),
    \end{aligned}
\end{equation}
where we have utilized the radial symmetry of \(\Wd\).
Finally, assuming that \(u\) is also the temperature field corresponding to the ``local'' vector flux \(\sigma\) for \(p\)-Laplacian, we get the formulae
\begin{equation}\label{eq:a3}
    \begin{aligned}
        \nabla u(x) &= -\kappa^{1-q}|\sigma(x)|^{q-2}\sigma(x),\qquad\text{and}\\
        \nabla u(x') &= -\kappa^{1-q}|\sigma(x')|^{q-2}\sigma(x').
    \end{aligned}
\end{equation}
It remains to substitute~\eqref{eq:a3} into~\eqref{eq:a2}, and the result subsequently into~\eqref{eq:a1} and note that
\(\kappa\kappa^{(p-1)(1-q)} = 1\) and \(|\sigma(x)|^{(p-1)(q-2)} = |\sigma(x)|^{2-p}\)
in order to arrive at~\eqref{eq:Fddot}.

\subsubsection{Rigorous analysis of \(\ddot{F}_\delta\)}
We will now show, that \(\ddot{F}_{\delta}\) possesses several properties, which are complementary to those enjoyed by \(R_{\delta}\).
Namely, we will show the following:
\begin{itemize}
  \item Antisymmetry: \(\ddot{F}_{\delta}\sigma \in L^q_a(\Od\times\Od)\), for each \(\sigma \in L^q(\O;\Rn)\).
  \item Norm-stability:
  \(
  \limsup_{\delta\searrow 0}
  \Id_{\delta}(\kappa;\ddot{F}_{\delta}\sigma)\leq \Idloc(\kappa;\sigma),
  \)
  for all sufficiently smooth \(\sigma\).
  \item Consistency:  \(\ddot{F}_{\delta}\sigma\in \Qd\) and
    \(\lim_{\delta\searrow 0}\|\diver \ddot{F}_{\delta}\sigma - \Div \sigma\|_{L^q(\O)} = 0,\)
    for all sufficiently smooth \(\sigma\).
\end{itemize}

Antisymmetry follows directly from the definition of \(\ddot{F}\), since
\begin{equation}\label{eq:antisym}
  \ddot{F}_{\delta}\sigma(x,x')=-\ddot{F}_{\delta}\sigma(x',x),
  \qquad \forall (x,x')\in \Rn\times\Rn.
\end{equation}
Furthermore, for each \(\sigma \in L^q(\O;\Rn)\) we can apply H{\"o}lder's inequality to show that \(\ddot{F}_{\delta}\sigma\in L^q(\Od\times\Od)\):
\[
  \begin{aligned}
    \|\ddot{F}_{\delta}\sigma\|_{L^q(\Od\times\Od)}
    \leq \left[\int_{\O} |\sigma(x)|^q \int_{\Od} |x-x'|^p \Wd^p(x-x')\dx'\dx \right]^{1/q} \leq K_{p,n}^{-1/q} \|\sigma\|_{L^q(\Od;\Rn)}.
  \end{aligned}
\]

We now proceed to establishing norm-stability, which is only marginally more difficult to obtain.
\begin{proposition}\label{prop:rstar2}
  Assume that \(\sigma\in C^1_c(\Rn;\Rn)\) and \(\kappa \in \kadmd\).
  Then
  \[
  \limsup_{\delta\searrow 0}
  \Id_{\delta}(\kappa;\ddot{F}_{\delta}\sigma)\leq \Idloc(\kappa;\sigma).
  \]
  \end{proposition}
  \begin{proof}
  We proceed with the direct computation. Let us put \(z=x'-x\). Then we can write:
  \[\begin{aligned}
    \Id_{\delta}(\kappa;\ddot{F}_{\delta}q)&=
    \int_{\Od}\kappa^{1-q}(x)\int_{\Od} |\ddot{F}_{\delta}\sigma(x,x')|^q \dx'\dx
  \\&\leq
  \frac{1}{2}\int_{\Od} \kappa^{1-q}(x)\int_{\Od}
  |\sigma(x)|^{q-p}|\sigma(x)\cdot (x-x')|^{p}\Wd^p(x-x')\dx'\dx
  \\&+
  \frac{1}{2}\int_{\Od} \kappa^{1-q}(x)\int_{\Od}
  |\sigma(x')|^{q-p}|\sigma(x')\cdot (x-x')|^{p}\Wd^p(x-x')\dx'\dx
  \\&=
  \frac{1}{2}\int_{\Od} \kappa^{1-q}(x)\underbrace{\int_{B(0,\delta)\cap \Od}
  |\sigma(x)|^q\left|\frac{\sigma(x)}{|\sigma(x)|}\cdot \frac{z}{|z|}\right|^{p}|z|^p\Wd^p(z)\,\mathrm{d}z}_{=E_1(x)}\dx
  \\&+
  \frac{1}{2}\int_{\Od} \kappa^{1-q}(x)\underbrace{\int_{B(0,\delta)\cap \Od}
  |\sigma(x+z)|^{q-p}\left|\sigma(x+z)\cdot \frac{z}{|z|}\right|^{p}|z|^p\Wd^p(z)\,\mathrm{d}z}_{=E_2(x)}\dx, 
  \end{aligned}
  \]
  where we have utilized the fact that \(\Wd\) is radial, the anti-symmetry of \(\ddot{F}_{\delta}\) and the convexity of \(|\cdot|^q\).
  Arguing as in~\eqref{eq:BBM0} we establish that \(E_1(x)=|\sigma(x)|^q\), for each \(x\in\O\).
  Our strategy now is to show that \(E_2(x)\) is close to \(E_1(x)\), in the sense that Lebesgue's dominated convergence theorem is applicable as \(\delta\searrow 0\).
  Since we know the value of \(E_1(x)\) only for \(x\in\O\), we will also need to show that the integrals of \(E_1\) and \(E_2\) over \(\Gamma_{\delta}\) are small.

  To this end we first state the upper bound
  \begin{equation}\label{eq:ubound_I2}\begin{aligned}
    |E_2(x)| &\leq \int_{B(0,\delta)}
    |\sigma(x+z)|^{q}\left|\frac{\sigma(x+z)}{|\sigma(x+z)|}\cdot \frac{z}{|z|}\right|^{p}|z|^p\Wd^p(z)\,\mathrm{d}z
    \\&\leq \|\sigma\|_{L^\infty(\Rn;\Rn)}^q\int_{B(0,\delta)}|z|^p\Wd^p(z)\,\mathrm{d}z=
    \|\sigma\|_{L^\infty(\Rn;\Rn)}^q K_{p,n}^{-1},
    \qquad \forall x\in \Rn,
  \end{aligned}
\end{equation}
  where we utilized H{\"o}lder's inequality.
  We now  consider the integrand
  \(g(r)=|\sigma(x+rs)|^{q-p}\left|\sigma(x+rs)\cdot s\right|^{p}\),
  where \(s\in \Sn\) and \(r\in \R\) are arbitrary.
  If \(\sigma(x)\neq 0\), the same holds over \(B(x,\delta)\), for small enough \(\delta>0\).  We can then directly compute
  \[\begin{aligned}
    g'(r) &= (q-p)|\sigma(x+rs)|^{q-p-1}\frac{\sigma(x+rs)^T \nabla \sigma(x+rs) s}{|\sigma(x+rs)|} \left|\sigma(x+rs)\cdot s\right|^{p}
    \\&+ p|\sigma(x+rs)|^{q-p}\left|\sigma(x+rs)\cdot s\right|^{p-1}
    s^T\nabla \sigma(x+rs) s,
  \end{aligned}\]
  for \(r \in (-\delta,\delta)\).
  Consequently, utilizing Cauchy--Bunyakovsky--Schwarz inequality we can estimate
  \[
    |g'(r)| \leq (|q-p|+p)\|\sigma\|_{L^\infty(\Rn;\Rn)}^{q-1} \|\nabla \sigma\|_{L^\infty(\Rn;\R^{n\times n})},
  \]
  and therefore owing to Taylor's theorem we have the inequality
  \[
    |g(r)-g(0)| \leq (|q-p|+p)\|\sigma\|_{L^\infty(\Rn;\Rn)}^{q-1} \|\nabla \sigma\|_{L^\infty(\Rn;\R^{n\times n})}|r|.
  \]
  If, on the other hand, \(\sigma(x)= 0\), then also \(g(0)=0\) and utilizing Cauchy--Bunyakovsky--Schwarz inequality and Taylor's theorem we can write
  \[
    |g(r)-g(0)| \leq |\sigma(x+rs)|^{q}
    \leq \|\nabla \sigma\|_{L^\infty(\Rn;\R^{n\times n})}^q |r|^q.
  \]
  Regardless of whether \(\sigma(x)=0\) or not, and having in mind that \(q>1\), for small \(|z|\) we can write the estimate
  \[
    |g(|z|)-g(0)| \leq C|z|,
  \]
  with the proportionality constant depending on \(p\), \(q\), \(\|\sigma\|_{L^\infty(\Rn;\Rn)}\) and \(\|\nabla \sigma\|_{L^\infty(\Rn;\R^{n\times n})}\).
  We now use \(s=z/|z|\) and \(r=|z|\) in the definition of \(g\) to write the following:
  \[
      \lim_{\delta\searrow 0}
      |E_2(x)-E_1(x)| \leq \lim_{\delta\searrow 0} \int_{B(0,\delta)}
      C|z|^{p+1}\Wd^p(z)\,\mathrm{d}z \leq \lim_{\delta\searrow 0} CK_{p,n}^{-1} \delta = 0,
  \]
 for each \(x\in \Od\).
 Since we have the upper bound~\eqref{eq:ubound_I2}, Lebesgue's dominated convergence theorem applies, implying the inequality
  \[\begin{aligned}
    \limsup_{\delta\searrow 0}
    \int_{\Od}\kappa^{1-q}(x)\int_{\Od} |\ddot{F}_{\delta}\sigma(x,x')|^q\dx'\dx
    &\leq \int_{\O} \kappa^{1-q}(x)|\sigma(x)|^q \dx \\&+ 
    \limsup_{\delta\searrow 0}\int_{\Gamma_{\delta}} \kappa^{1-q}(x)\int_{\Od}|\ddot{F}_{\delta}\sigma(x,x')|^q \dx'\dx,
  \end{aligned}\]
  with the last limit being \(0\) owing to the continuity of Lebesgue's integral.
  In summary, we can conclude that
  \[
  \limsup_{\delta\searrow 0}
  \Id_{\delta}(\kappa;\ddot{F}_{\delta}\sigma)\leq \Idloc(\kappa;\sigma).\qedhere
  \]
\end{proof}

We postpone the analysis of consistency of \(\ddot{F}_\delta\) to Section~\ref{sec:Fconsist}.
For now, we only state the final result without a proof.
\begin{proposition}\label{prop:Fconsist}
  Assume that \(\sigma\in C^2_c(\Rn;\Rn)\).
  Then \(\ddot{F}_{\delta}\sigma \in \Qd\),  and \(\lim_{\delta\searrow 0} \|\diver\ddot{F}_{\delta}\sigma-\Div\sigma\|_{L^q(\O)} = 0\).
\end{proposition}

\subsubsection{Application to the lower estimate for the local problem and \(\Gamma\)-convergence}
With these results at our disposal, we can prove the following claim.
\begin{proposition}\label{thm:nonlocal_approximation_2sided}
  Assume that \(\kappa \in \kadmd\) and \(\sigma\in \Q(f)\) are given.
  Then there is a family \((\qnl_{\delta})\), such that \(\qnl_\delta \in \Qd(f)\), and
  \(\limsup_{\delta\searrow 0} \Id_{\delta}(\kappa;\qnl_{\delta}) \leq
  \Idloc(\kappa;\sigma)\).
  \end{proposition}
  \begin{proof}
  Let us fix an arbitrary \(\epsilon \in (0,1/2)\).
  Since \(C^\infty_c(\Rn;\Rn)\) is dense in \(\Q\) (cf.~\cite[Theorem 2.4]{girault_raviart}),
  we can select \(\sigma_{\epsilon} \in C^\infty_c(\Rn;\Rn)\) such that
  \(\|\sigma_\epsilon-\sigma\|_{\Q}<\epsilon\). In particular,
  \(f_\epsilon = \Div \sigma_\epsilon\) satisfies the estimate \(\|f_\epsilon-f\|_{L^q(\O)}<\epsilon\),
  and we have the equality
  \begin{equation}\label{eq:idloc_bound}
    \lim_{\epsilon\searrow 0}\Idloc(\kappa;\sigma_\epsilon) = \Idloc(\kappa;\sigma),
  \end{equation}
  since \(\Idloc(\kappa;\cdot)\) is continuous with respect to strong convergence in \(L^q(\O)\).
  Let us put \(\qnl_{\delta,\epsilon} = \ddot{F}_{\delta}\sigma_{\epsilon}\),
  and \(f_{\delta,\epsilon} = \diver \qnl_{\delta,\epsilon}\).
  According to Proposition~\ref{prop:Fconsist}, we can select \(\hat{\delta}(\epsilon)>0\),
  such that for all \(\delta \in (0,\hat{\delta}(\epsilon))\) we have the bound:
  \[\|f_{\delta,\epsilon} - f\|_{L^q(\O)}
  \leq \|\diver \qnl_{\delta,\epsilon} - f_{\epsilon}\|_{L^q(\O)}
  + \|f_{\epsilon} - f\|_{L^q(\O)}
  < \|\diver \ddot{F}_{\delta} \sigma_{\epsilon} - \Div \sigma_{\epsilon}\|_{L^q(\O)} + \epsilon < 2\epsilon.\]
  Let \(\qnl_{\delta}\in \Qd(f)\) and \(\pnl_{\delta,\epsilon} \in \Qd(f_{\delta,\epsilon})\) be the solutions to~\eqref{eq:nonlocal_dual}
  corresponding to the two different volumetric heat sources \(f\) and \(f_{\delta,\epsilon}\).
  In view of the stability estimate, Corollary~\ref{cor:Lip_stab}, we have
  the following estimate:
  \[\begin{aligned}
    \Id_{\delta}(\kappa;\qnl_{\delta})
    \leq
    \Id_{\delta}(\kappa;\pnl_{\delta,\epsilon})
    + L_\epsilon\|f-f_{\delta,\epsilon}\|_{L^q(\Od)}\leq 
    \Id_{\delta}(\kappa;\pnl_{\delta,\epsilon})+2\epsilon L_\epsilon,
  \end{aligned}\]
  where \(L_\epsilon\) is continuous and nondecreasing with respect to \(\|f_{\delta,\epsilon}\|_{L^q(\Od)}\leq \|f\|_{L^q(\O)}+2\epsilon\),
  and therefore remains bounded for all small \(\epsilon>0\).
  Consequently, we can write
  \[\begin{aligned}
  \limsup_{\delta \searrow 0}\Id_{\delta}(\kappa;\qnl_{\delta}) &\leq
  \limsup_{\delta \searrow 0}
  \Id_{\delta}(\kappa;\pnl_{\delta,\epsilon}) + 2\epsilon L_\epsilon
  \leq
  \limsup_{\delta \searrow 0}
  \Id_{\delta}(\kappa;\qnl_{\delta,\epsilon}) + 2\epsilon L_\epsilon
   \leq \Idloc(\kappa;\sigma_\epsilon) +2\epsilon L_\epsilon,
  \end{aligned}\]
  where we have utilized the optimality of \(\pnl_{\delta,\epsilon}\) as well
  as Proposition~\ref{prop:rstar2}. 
  Finally, it remains to let \(\epsilon\searrow 0\) and utilize~\eqref{eq:idloc_bound} to arrive at the desired claim:
  \[
    \limsup_{\delta \searrow 0}\Id_{\delta}(\kappa;\qnl_{\delta})
    \leq \Idloc(\kappa;\sigma).\qedhere
  \]
  \end{proof}
  
  \begin{corollary}\label{cor:gamma_limsup}
    For each \(\kappa \in \kadmd\) we have the inequality
    \[
      \limsup_{\delta\searrow 0} \hat{\imath}_{\delta}(\kappa)
      \leq \hat{\imath}_{\text{loc}}(\kappa).
    \]
    In particular, the optimal values of the local~\eqref{eq:local_min} and nonlocal~\eqref{eq:nl_min} problems satisfy the inequality
    \[\limsup_{\delta\searrow 0} d_{\delta} \leq d_{\text{loc}}.\]
  \end{corollary}
  \begin{proof}
    The first claim follows from applying Proposition~\ref{thm:nonlocal_approximation_2sided} to 
    \(\sigma = \argmin_{\tau\in\Q(f)}\hat{I}_{\text{loc}}(\kappa;\tau)\),
    while the second follows from applying the same result to a solution to~\eqref{eq:nl_min}.
  \end{proof}
  
  Now, as a direct consequence of Proposition~\ref{prop:liminf} and Corollary~\ref{cor:gamma_limsup} we obtain the following nonlocal-to-local approximation result.
  
  \begin{theorem}\label{thm:primal_summary1}
    The family of objective functions \(i_{\delta}(\cdot)\)  \(\Gamma\)-converges, as \(\delta\searrow 0\), towards \(i_{\text{loc}}(\cdot)\) on the set
    \(\kadm\subset L^\infty(\O)\), considered with its weak\(^*\) topology.
  \end{theorem}
    
  As a direct consequence of \(\Gamma\)-convergence, we obtain the convegrence of optimal values, namely
  \[d_{\text{loc}}=\lim_{\delta\searrow 0} d_{\delta},\]
  and consequently owing to the strong duality also
  \[
    p_{\text{loc}}
    =
    \lim_{\delta\searrow 0}
    p_{\delta},
  \]
  see~\cite{braides2006handbook}.
  Furthermore, in view of a priori bounds on the optimal solutions we have established, we also get their convergence in the following sense.
  Let \(\{(\kappa_\delta,\qnl_\delta)\}\) be a family of optimal solutions to~\eqref{eq:nl_min} as \(\delta\searrow 0\).
  Then, the family \(\{(\kappa_\delta,R_\delta \qnl_\delta)\}\) is sequentially compact, with respect to the weak\(^*\) topology of \(L^\infty(\O)\)
  and  the weak topology of \(L^q(\O;\Rn)\).
  Each limit point of \(\{(\kappa_\delta,R_\delta \qnl_\delta)\}\) as \(\delta\searrow 0\) is an optimal solution to~\eqref{eq:local_min}.

\section{Consistency of \(\ddot{F}_\delta\)}\label{sec:Fconsist}
We now proceed to establishing the consistency of \(\ddot{F}_\delta\), that is, we will prove Proposition~\ref{prop:Fconsist}.
Owing to the non-linearity of \(\ddot{F}_\delta\) for \(p\neq 2\), the proofs are somewhat longer than in the quadratic case \(p=2\).
We begin with the following lemma, which will be applied to Taylor polynomials of \(\diver\ddot{F}_\delta\sigma\) later on.

\begin{lemma}\label{lem:crazytrace}
  Let \(e\in \Sn\) and \(A\in \R^{n\times n}\) be arbitrary.
  Then
\begin{equation}\label{eq:traceA}
  |\Sn|^{-1}
  \int_{\Sn}
  \left|e^T s\right|^{p-2}
  s^T\left[(p-1)I  + (2-p)ee^T\right]As
  \,\mathrm{d}s = K_{p,n} \trace A.
\end{equation}
\end{lemma}
\begin{proof}
  Let \(Q\in \R^{n\times n}\) be an orthogonal transformation mapping \(e\) to some other unit vector \(\tilde{e}=Qe\in \Sn\).
  Then the left hand side of~\eqref{eq:traceA} reduces to
  \begin{equation}\label{eq:crazytrace0}
    |\Sn|^{-1}
    \int_{\Sn}
    \left|\tilde{e}^T \tilde{s}\right|^{p-2}
    \tilde{s}^T\left[(p-1)I  + (2-p)\tilde{e}\tilde{e}^T\right]QAQ^T\tilde{s}
    \,\mathrm{d}\tilde{s},
  \end{equation}
  where we have utilized the variable substitution \(\tilde{s} = Qs\).
  The right hand side does not change, since \(\trace A = \trace(QAQ^T)\).
  Therefore, without any loss of generality we may assume that \(e=e_n\),
  where \(e_1,\dots,e_n\in\Sn\) are the standard basis vectors in \(\Rn\).

  Then
  \begin{equation}\label{eq:crazytrace1}\begin{aligned}
    \int_{\Sn}\left|e_n^T s\right|^{p-2} [e_n^T s]e_n^T A s\,\mathrm{d}s
    &=
    \sum_{i,j=1}^n A_{ij} \int_{\Sn}\left|e_n^T s\right|^{p-2} [e_n^T s][e_n^T e_i] [e_j^T s]\,\mathrm{d}s
    \\&= A_{nn}\int_{\Sn}\left|e_n^T s\right|^{p}\,\mathrm{d}s
    = |\Sn|K_{p,n}A_{nn},
  \end{aligned}\end{equation}
  with integrals corresponding to \(i\neq n\) or \(j\neq n\) evaluating to \(0\) owing to the orthogonality and symmetry considerations.

  Furthermore
  \begin{equation}\label{eq:crazytrace2}\begin{aligned}
    \int_{\Sn}\left|e_n^T s\right|^{p-2} s^T A s\,\mathrm{d}s
    &=
    \sum_{i,j=1}^n A_{ij} \int_{\Sn}\left|e_n^T s\right|^{p-2} [e_i^T s][e_j^T s]\,\mathrm{d}s
    = \sum_{i=1}^{n-1}E_{i,n}A_{ii} + |\Sn|K_{p,n}A_{nn},
  \end{aligned}\end{equation}
  where we denoted the value of the integral corresponding to \(i=j\neq n\) by \(E_{i,n}\) and used the fact that the integrals corresponding to \(i\neq j\) evaluate to \(0\).
  Owing to the symmetry, \(E_{1,n}= E_{2,n}=\dots = E_{n-1,n}\), so we only need to compute one of them.

  Let us directly evaluate \(E_{n-1,n}\) using the spherical coordinates
  \[\begin{aligned}
    s_1 &= \cos(\phi_1),  &\qquad  \phi_1 &\in [0,\pi],\\
    s_2 &= \sin(\phi_1)\cos(\phi_2),  &\qquad  \phi_2 &\in [0,\pi],\\
    &\vdots  &\qquad &\vdots\\
    s_{n-1} &= \sin(\phi_1)\cdots\sin(\phi_{n-2})\cos(\phi_{n-1}),
    &\qquad  \phi_{n-1} &\in [0,2\pi],\\
    s_{n} &= \sin(\phi_1)\cdots\sin(\phi_{n-2})\sin(\phi_{n-1}),
  \end{aligned}\]
  with \(\mathrm{d}s = \sin^{n-2}(\phi_1)\sin^{n-3}(\phi_2)\dots\sin(\phi_{n-2})\,\mathrm{d}\phi_1\,\mathrm{d}\phi_2\cdots\,\mathrm{d}\phi_{n-1}\).
  The integral becomes
  \begin{equation}\label{eq:crazytrace3}\begin{aligned}
    &E_{n-1,n} = \int_{\Sn}\left|s_n\right|^{p-2} s_{n-1}^2\,\mathrm{d}s\\
    &=\int_{0}^{2\pi}\cdots\int_{0}^{2\pi}\int_{0}^{\pi} \left|\prod_{i=1}^{n-2} \sin(\phi_{i})\right|^{p} \left[\left|\sin(\phi_{n-1})\right|^{p-2}\cos^2(\phi_{n-1})\right] \prod_{i=1}^{n-2} \sin^{n-1-i}(\phi_{i}) \,\mathrm{d}\phi_1\,\mathrm{d}\phi_2\cdots\,\mathrm{d}\phi_{n-1}\\
    &= \int_{\Sn}\left|s_n\right|^{p}\,\mathrm{d}s
    \left[
      \int_{0}^{2\pi}|\sin(\phi_{n-1})|^p\mathrm{d}\phi_{n-1}
    \right]^{-1}
    \left[
    \int_{0}^{2\pi}|\sin(\phi_{n-1})|^{p-2} \cos^2(\phi_{n-1})\mathrm{d}\phi_{n-1}\right]\\
    \\&=
    |\Sn|K_{p,n}
    \left[
      4\int_{0}^{\pi/2}\sin^p(\phi_{n-1})\mathrm{d}\phi_{n-1}
    \right]^{-1}
    \left[
    \frac{4}{p-1}\int_{0}^{\pi/2}[\sin^{p-1}(\phi_{n-1})]'\cos(\phi_{n-1})\mathrm{d}\phi_{n-1}\right]\\
    &=\frac{1}{p-1}|\Sn|K_{p,n}
    \left[
      \int_{0}^{\pi/2}\sin^p(\phi_{n-1})\mathrm{d}\phi_{n-1}
    \right]^{-1}
    \left[
    \int_{0}^{\pi/2}\sin^{p}(\phi_{n-1})\mathrm{d}\phi_{n-1}\right]=\frac{1}{p-1}|\Sn|K_{p,n}.
  \end{aligned}
 \end{equation}
  Combining \eqref{eq:crazytrace1},
  \eqref{eq:crazytrace2}, and \eqref{eq:crazytrace3} yields the desired claim~\eqref{eq:traceA}.
\end{proof}

A simple consequence of Lemma~\ref{lem:crazytrace} and the normalization condition~\eqref{eq:normalization} is the following statement about volumetric integrals related to~\eqref{eq:traceA}.

\begin{corollary}\label{cor:crazytrace}
  Let \(e\in \Sn\) and \(A\in \R^{n\times n}\) be arbitrary.
  Then
\begin{equation}\label{eq:traceAV}
  \lim_{\epsilon\searrow 0}\int_{B(0,\delta)\setminus B(0,\epsilon)}
  \left|e^T z\right|^{p-2}
  z^T\left[(p-1)I  + (2-p)ee^T\right]Az\Wd^p(z)
  \,\mathrm{d}z = \trace A.
\end{equation}
\end{corollary}
\begin{proof}
  We evaluate the integral~\eqref{eq:traceAV} in spherical coordinates:
  \[\begin{aligned}
    &\lim_{\epsilon\searrow 0}\int_{B(0,\delta)\setminus B(0,\epsilon)}
    \left|e^T z\right|^{p-2}
    z^T\left[(p-1)I  + (2-p)ee^T\right]Az\Wd^p(z)
    \,\mathrm{d}z \\&= 
    \lim_{\epsilon\searrow 0}\int_{\epsilon}^{\delta} r^{n-1}r^p\Wd^p(r)
    \int_{\Sn} \left|e^T s\right|^{p-2}
    s^T\left[(p-1)I  + (2-p)ee^T\right]As\,\mathrm{d}s\,\mathrm{d}r
    \\&=   K_{p,n}  \lim_{\epsilon\searrow 0}\int_{\epsilon}^{\delta} r^{n-1}r^p\Wd^p(r) |\Sn|\,\mathrm{d}r\,\trace A
    =  K_{p,n} \lim_{\epsilon\searrow 0}\int_{B(0,\delta)\setminus B(0,\epsilon)}|z|^p\Wd^p(z)\,\mathrm{d}z\,\trace A = \trace A,
  \end{aligned}\]
  owing to Lemma~\ref{lem:crazytrace} and the normalization condition~\eqref{eq:normalization}.
\end{proof}

The consistency of \(\ddot{F}_{\delta}\sigma\) hinges on the behaviour of the integrals
\begin{align}\label{eq:I_sigma_delta}
  J_{p,\delta,\epsilon}[\sigma](x) &=   
  \int_{B(0,\delta)\setminus B(0,\epsilon)}
    j_{p,\sigma,z/|z|}(x+z)|z|^{p-1}\Wd^p(z)\,\mathrm{d}z,\qquad\text{where}\\
    j_{p,\sigma,s}(x) &= |\sigma(x)|^{2-p}|\sigma(x)\cdot s|^{p-2}[\sigma(x)\cdot s],
\end{align}
for smooth enough \(\sigma\), where \(s \in \Sn\) is an arbitrary unit vector. The reason for this is as follows.

\begin{proposition}\label{prop:rstar1}
  Let \(\sigma\) be such that the family \(J_{p,\delta,\epsilon}[\sigma]\) is dominated by some \(L^q(\O)\)-function, uniformly with respect to \(\epsilon>0\).
  Then \(\ddot{F}_{\delta}\sigma\in \Qd\) and
  \[\diver \ddot{F}_{\delta}\sigma(x)= \lim_{\epsilon\searrow 0}J_{p,\delta,\epsilon}[\sigma](x).\]
\end{proposition}
\begin{proof}
  Let us take an arbitrary \(u \in \Udz\).
  We begin with the integration by parts:
  \[
    \begin{aligned}
      &\int_{\Od}\int_{\Od}\ddot{F}_{\delta}\sigma(x,x')\grad u(x,x')\dx'\dx
      \\&=
      \frac{1}{2}\int_{\Od}\int_{\Od} [u(x)-u(x')]
      |\sigma(x)|^{2-p}|\sigma(x)\cdot (x-x')|^{p-2}[\sigma(x)\cdot (x-x')]\Wd^p(x-x')\dx'\dx
      \\&+
      \frac{1}{2}\int_{\Od}\int_{\Od} [u(x)-u(x')]
      |\sigma(x')|^{2-p}|\sigma(x')\cdot (x-x')|^{p-2}[\sigma(x')\cdot (x-x')]\Wd^p(x-x')\dx'\dx
      \\&=
      \lim_{\epsilon\searrow 0}\bigg\{\frac{1}{2}
      \iint\limits_{O_\epsilon}
      [u(x)-u(x')]
      |\sigma(x)|^{2-p}|\sigma(x)\cdot (x-x')|^{p-2}[\sigma(x)\cdot (x-x')]\Wd^p(x-x')\dx'\dx
      \\&
      \phantom{\lim_{\epsilon\searrow 0}\bigg\{}+\frac{1}{2}
      \iint\limits_{O_\epsilon}
      [u(x)-u(x')]
      |\sigma(x')|^{2-p}|\sigma(x')\cdot (x-x')|^{p-2}[\sigma(x')\cdot (x-x')]\Wd^p(x-x')\dx'\dx\bigg\}
      \\&=
      -\lim_{\epsilon\searrow 0}\bigg\{
      \int_{\O}u(x)\underbrace{\int_{B(0,\delta)\setminus B(0,\epsilon)} 
      |\sigma(x)|^{2-p}|\sigma(x)\cdot z|^{p-2}[\sigma(x)\cdot z]\Wd^p(z)\,\mathrm{d}z}_{=0}\dx
      \\&\phantom{=}
      \phantom{\lim_{\epsilon\searrow 0}\bigg\{}+
      \int_{\O}u(x)\int_{B(0,\delta)\setminus B(0,\epsilon)} 
      |\sigma(x+z)|^{2-p}|\sigma(x+z)\cdot z|^{p-2}[\sigma(x+z)\cdot z]\Wd^p(z)\,\mathrm{d}z\dx\bigg\}
      \\&=
      -\lim_{\epsilon\searrow 0}
      \int_{\O}u(x) J_{p,\delta,\epsilon}[\sigma](x)\dx,
    \end{aligned}
  \]
  where \(O_\epsilon = \{\, (x,x')\in\Od\times\Od \colon |x-x'|>\epsilon \,\}\), and \(z=x'-x\). The steps rely primarily on (anti-)symmetries of the terms, including the radial symmetry of \(\Wd(\cdot)\), and Fubini's theorem.
  Finally, we can move the limit under the integral sign owing to the assumed uniform boundedness and Lebesgue's dominated convergence theorem.
\end{proof}  

For \(p=2\) the nonlinearity in \(\ddot{F}_{\delta}\) disappears, and \(\|\lim_{\delta\searrow 0}\lim_{\epsilon\searrow 0} J_{2,\delta,\epsilon}[\sigma] - \Div\sigma\|_{L^2(\O;\Rn)}=0\), \(\forall \sigma \in C^2_c(\Rn;\Rn)\), and small \(\delta>0\),
see for example~\cite[Proposition 6.1]{evgrafov2021dual}.
We proceed with the analysis of the two remaining cases \(p>2\) and \(1<p<2\), starting with a simpler case \(p>2\).
In particular, in both cases we will show that \(J_{p,\delta,\epsilon}[\sigma]\) is bounded and converges pointwise to \(\Div\sigma\) for smooth enough \(\sigma\), therefore Lebesgue dominated convergence theorem yields the desired conclusion.  In other words, in the remainder of the paper we will check that the prerequisites of Proposition~\ref{prop:rstar1} hold for \(\sigma \in C^2_c(\Rn;\Rn)\), which in turn is sufficient to conclude that Proposition~\ref{prop:Fconsist} holds.

\subsection{Analysis of \(J_{p,\delta,\epsilon}[\sigma](\cdot)\) for \(p>2\)}

We begin by establishing the uniform boundedness of \(J_{p,\delta,\epsilon}[\sigma](\cdot)\).

\begin{proposition}\label{prop:i_lip}
  Assume that \(\sigma\in C^1_c(\Rn;\Rn)\) and \(s\in \Sn\).
  Then \(j_{p,\sigma,s}\) is Lipschitz continuous with constant
  \(L_{p,\sigma} = (2p-3) \|\nabla\sigma\|_{L^\infty(\Rn;\R^{n\times n})}\).
\end{proposition}
\begin{proof}
  Let \(y_0,y_1 \in \Rn\) be arbitrary, and let \(y_t = y_0+t(y_1-y_0)\), 
  \(t\in [0,1]\). We consider two cases.

  {\bfseries Case 1}: \(\sigma(y_t)\neq 0\), \(\forall t\in [0,1]\).
  Then \(j_{p,\sigma,s}\) is continuously differentiable in the vicinity of the line segment between \(y_1\) and \(y_0\).
  We can therefore apply Taylor's theorem:
  \begin{equation}\label{eq:taylor_i}
      \begin{aligned}
        j_{p,\sigma,s}(y_1) - j_{p,\sigma,s}(y_0)
        &=
        \int_{0}^1 
        \bigg\{
        (2-p)\left[
          |\sigma(y_t)|^{1-p}\frac{\sigma(y_t)^T}{|\sigma(y_t)|}\nabla\sigma(y_t) [y_1-y_0]
          \left|\sigma^T(y_t)s\right|^{p-2}
          \left[\sigma^T(y_t)s\right]      
        \right]
        \\&\phantom{+ \int_{0}^1} +(p-1)
        \left[
          |\sigma(y_t)|^{2-p} 
        \left|\sigma^T(y_t) s\right|^{p-2} s^T \nabla\sigma(y_t) [y_1-y_0]
        \right]
        \bigg\}\,\mathrm{d}t
        \\&=\int_{0}^1 
        \bigg\{
        (2-p)\left[
          \frac{\sigma(y_t)^T}{|\sigma(y_t)|}\nabla\sigma(y_t) [y_1-y_0]\right]
          \left|\frac{\sigma^T(y_t)}{|\sigma(y_t)|}s\right|^{p-2}
          \left[\frac{\sigma^T(y_t)}{|\sigma(y_t)|}s\right]      
        \\&\phantom{+ \int_{0}^1} +(p-1)
        \left|\frac{\sigma^T(y_t)}{|\sigma(y_t)|} s\right|^{p-2} 
        \left[s^T \nabla\sigma(y_t) [y_1-y_0]
        \right]\bigg\}\,\mathrm{d}t.
      \end{aligned}
    \end{equation}
  Consequently,
  \[\begin{aligned}
    \left|j_{p,\sigma,s}(y_1) - j_{p,\sigma,s}(y_0)\right| \leq L|y_1-y_0|,
  \end{aligned}\]
  where we can take
  \[\begin{aligned}
    L &= [(p-1)+|2-p|]\sup_{t\in[0,1]} |\nabla\sigma(y_t)|
    \leq (2p-3)\|\nabla\sigma\|_{L^\infty(\Rn;\R^{n\times n})}.
  \end{aligned}\]

  {\bfseries Case 2}: \(\sigma(y_{\hat{t}})=0\), for some \(\hat{t}\in [0,1]\).
  Then we apply Taylor's theorem to \(\sigma\) around \(y_{\hat{t}}\) to get the following estimate:
  \begin{equation}\label{eq:Taylor_i0}
    |j_{p,\sigma,s}(y_t)- j_{p,\sigma,s}(y_{\hat{t}})|
    =|j_{p,\sigma,s}(y_t)|
    \leq
    \left|\sigma(y_t)\right| =\left|\int_{\hat{t}}^{t}\nabla \sigma(y_\tau)[y_1-y_0]\right|\,\mathrm{d}\tau.    
  \end{equation}
  Consequently
  \[\begin{aligned}
    |j_{p,\sigma,s}(y_1)- j_{p,\sigma,s}(y_0)|
  &\leq |j_{p,\sigma,s}(y_1)- j_{p,\sigma,s}(y_{\hat{t}})|
  +|j_{p,\sigma,s}(y_0)- j_{p,\sigma,s}(y_{\hat{t}})|
  \\&\leq \sup_{t\in[0,1]} |\nabla\sigma(y_t)| [|1-\hat{t}|+|\hat{t}-0|]|y_1-y_0| \\&= \sup_{t\in[0,1]} |\nabla\sigma(y_t)||y_1-y_0|.
  \end{aligned}\qedhere\]
\end{proof}

\begin{corollary}
  Assume that \(\sigma\in C^1_c(\Rn;\Rn)\).
  Then 
  \begin{equation}\label{eq:Idelta_bound}
    |J_{p,\delta,\epsilon}[\sigma](x)| \leq K_{p,n}^{-1}L_{p,\sigma}, \qquad \forall 0<\epsilon<\delta, x\in \Rn.
  \end{equation}
\end{corollary}
\begin{proof}
  \[
    \begin{aligned}
  |J_{p,\delta,\epsilon}[\sigma](x)| &= 
  \left|\int_{B(0,\delta)\setminus B(0,\epsilon)}
  j_{p,\sigma,z/|z|}(x+z)|z|^{p-1}\Wd^p(z)\,\mathrm{d}z
  \right|
  \\&\leq  
  \bigg|\underbrace{\int_{B(0,\delta)\setminus B(0,\epsilon)}
  j_{p,\sigma,z/|z|}(x)|z|^{p-1}\Wd^p(z)\,\mathrm{d}z}_{=0}
  \bigg|\\
  &+\int_{B(0,\delta)\setminus B(0,\epsilon)}
  \underbrace{|j_{p,\sigma,z/|z|}(x+z)-j_{p,\sigma,z/|z|}(x)|}_{\leq L_{p,\sigma}|z|}|z|^{p-1}\Wd^p(z)\,\mathrm{d}z
  \\&\leq K_{p,n}^{-1}L_{p,\sigma}.
  \end{aligned}\qedhere\]
\end{proof}

We will now analyze the pointwise behaviour of \(J_{p,\delta,\epsilon}[\sigma]\).
\begin{proposition}\label{prop:consistency_pgt2}
  Assume that \(\sigma \in C^2_c(\Rn;\Rn)\).
  Then
  \begin{equation}\label{eq:cons_pgt2_lim}
    \lim_{\delta\searrow 0}\lim_{\epsilon\searrow 0} J_{p,\delta,\epsilon}[\sigma](x) = \Div\sigma(x), \qquad \text{for almost all \(x\in \Rn\)}.
  \end{equation}
\end{proposition}
\begin{proof}
  We consider three cases.

  {\bfseries Case 1}: \(\sigma(x)\neq 0\). 
  Let \(\hat{\delta}(x) = |\sigma(x)|/\|\nabla\sigma\|_{L^\infty(\Rn;\R^{n\times n})}\).
  Then for all \(0<\delta < \hat{\delta}(x)/2\) and for all  \(z 
  \in B(0,\delta)\) we have 
  \begin{equation}\label{eq:lbound_sigma}|\sigma(x+z)|\geq |\sigma(x)|-|z|\|\nabla\sigma\|_{L^\infty(\Rn;\R^{n\times n})} > |\sigma(x)|/2.\end{equation}
  We apply Taylor's theorem as in~\eqref{eq:taylor_i} with \(y_0 = x\) and \(y_1 = x+z\), let \(s = z/|z|\), and integrate each of the terms.
  For the constant term we get
  \begin{equation}\label{eq:Taylor0}
    \int_{B(0,\delta)\setminus B(0,\epsilon)}
    j_{p,\sigma,z/|z|}(x) |z|^{p-1}\Wd^p(z)\,\mathrm{d}z = 0,
  \end{equation}
  owing to the symmetry.
  For one of the first order terms we proceed as follows:
  \begin{equation}\label{eq:Taylor11}\begin{aligned}
    &\int\limits_{B(0,\delta)\setminus B(0,\epsilon)}\int_{0}^1
    \underbrace{\left[\frac{\sigma(x+tz)^T}{|\sigma(x+tz)|}\nabla\sigma(x+tz)\frac{z}{|z|}\right]}_{\text{Lipschitz}}
    \underbrace{\left|\frac{\sigma(x+tz)}{|\sigma(x+tz)|}\cdot\frac{z}{|z|}\right|^{p-2}
      \left[\frac{\sigma(x+tz)}{|\sigma(x+tz)|}\cdot\frac{z}{|z|}\right]}_{\text{Lipschitz}}
      \,\mathrm{d}t|z|^{p}\Wd^p(z)\,\mathrm{d}z
    \\&=
    \int_{0}^1\int\limits_{B(0,\delta)\setminus B(0,\epsilon)}
    \left[\frac{\sigma(x)^T}{|\sigma(x)|}\nabla\sigma(x)\frac{z}{|z|}\right]
      \left|\frac{\sigma(x)}{|\sigma(x)|}\cdot\frac{z}{|z|}\right|^{p-2}
      \left[\frac{\sigma(x)}{|\sigma(x)|}\cdot\frac{z}{|z|}\right]
      |z|^{p}\Wd^p(z)\,\mathrm{d}z\,\mathrm{d}t + O(\delta)
      \\&=
      \int\limits_{B(0,\delta)\setminus B(0,\epsilon)}
      \left[\frac{\sigma(x)^T}{|\sigma(x)|}\nabla\sigma(x)\cdot \frac{z}{|z|}\right]
        \left|\frac{\sigma(x)}{|\sigma(x)|}\cdot\frac{z}{|z|}\right|^{p-2}
        \left[\frac{\sigma(x)}{|\sigma(x)|}\cdot\frac{z}{|z|}\right]
        |z|^{p}\Wd^p(z)\,\mathrm{d}z + O(\delta),
      \end{aligned}\end{equation}
    because the integrand is a product of at least once continuously differentiable functions with respect to \(x\) in \(B(0,\delta)\).
    Similarly
    \begin{equation}\label{eq:Taylor12}\begin{aligned}
      &\int\limits_{B(0,\delta)\setminus B(0,\epsilon)}
      \int_{0}^1
      \underbrace{\left|\frac{\sigma(x+tz)}{|\sigma(x+tz)|}\cdot\frac{z}{|z|}\right|^{p-2}}_{\text{H{\"o}lder}}
      \underbrace{\left[\frac{z^T}{|z|}\nabla\sigma(x+tz)\frac{z}{|z|}\right]}_{\text{Lipschitz}}
      \,\mathrm{d}t
      |z|^{p}\Wd^p(z)\,\mathrm{d}z
      \\&=
      \int\limits_{B(0,\delta)\setminus B(0,\epsilon)}
      \left|\frac{\sigma(x)}{|\sigma(x)|}\cdot\frac{z}{|z|}\right|^{p-2}
      \left[\frac{z^T}{|z|}\nabla\sigma(x)\frac{z}{|z|}\right]
      |z|^{p}\Wd^p(z)\,\mathrm{d}z + O(\delta^{\min\{1,p-2\}}),  
      \end{aligned}\end{equation}
      because the integrand is a product of a \(\min\{1,p-2\}\)-H{\"o}lder continuous function and a once continuously differentiable one with respect to \(x\) in \(B(0,\delta)\).
    It remains to utilize Corollary~\ref{cor:crazytrace} with \(e = \frac{\sigma(x)}{|\sigma(x)|}\), \(A=\nabla\sigma(x)\), and let \(\delta\searrow 0\).

    {\bfseries Case 2}: \(\sigma(x)=0\), \(\nabla \sigma(x)=0\).
    By utilizing Taylor's formula of degree two, we improve~\eqref{eq:Taylor_i0} to
    \[|j_{p,\sigma,s}(x+z)|
    \leq
    \left|\sigma(x+z)\right| =
    \frac{1}{2}\left|\sum_{i,j}^n\int_{0}^{1}
    \partial^2_{ij}\sigma(x+tz)z_iz_j t^2\,\mathrm{d}t\right| = O(|z|^2),
    \]
    with constants depending only on the second order derivatives of \(\sigma\).
    After integration this yields the estimate
      \[\lim_{\delta\searrow 0}\lim_{\epsilon\searrow 0} J_{p,\delta,\epsilon}[\sigma](x)= \lim_{\delta\searrow 0} O(\delta) = 0.
      \]
      Therefore, in this case~\eqref{eq:cons_pgt2_lim} also holds.

      {\bfseries Case 3}: \(\sigma(x)=0\), \(\nabla \sigma(x)\neq 0\).
      Let \(G = \{\, x\in \Rn: \sigma(x)=0 \text{\ and\ }
      \nabla \sigma(x)\neq 0 \,\}\subseteq \supp \sigma\).
      Then \[G\subseteq 
      \cup_{i,j=1}^n\cup_{k\in \N} G_{ijk} = \cup_{i,j=1}^n \cup_{k\in \N} \left\{\, x\in \Rn: \sigma_i(x)=0 \text{\ and\ }
      \frac{1}{k}\leq \left|\frac{\partial \sigma_i}{\partial x_j}(x)\right|< \frac{1}{k-1} \,\right\}.
      \]
      Finally, at each point of \(G_{ijk}\), the implicit function theorem applies, allowing us to express \(x_j\) as a \(C^1\) function of the rest of coordinates with the derivative bounded by \(O(k \|\nabla\sigma(x)\|_{L^\infty(\Rn;\R^{n\times n})})\). The Lebesgue measure of the graph of a such a function is \(0\).
\end{proof}

\subsection{Analysis of \(J_{p,\delta,\epsilon}[\sigma]\) for \(1<p<2\)}

We begin with some additional notation.
In the proof of case 1 of Proposition~\ref{prop:consistency_pgt2}, we put 
\(\hat{\delta}(x) = |\sigma(x)|/\|\nabla\sigma\|_{L^\infty(\Rn;\R^{n\times n})}\)
and when needed considered \(0<\delta < \hat{\delta}(x)/2\) in order to have a 
uniform positive  lower bound~\eqref{eq:lbound_sigma} on \(|\sigma(x+z)|\), 
for all \(z\in B(0,\delta)\).
In the present case, we will also need a uniform lower bound on the absolute value of 
the inner product \(\sigma(x+z)\cdot s\), \(z\in B(0,\delta)\), \(s\in\Sn\).
With this in mind we estimate this quantity as
\begin{equation}\label{eq:lbdot}
  |\sigma(x+z)\cdot s| = \left|\sigma(x)\cdot s + s^T \int_0^1 \nabla \sigma(x+tz)z\,\mathrm{d}t\right| \geq |\sigma(x)\cdot s| - |z|\|\nabla \sigma\|_{L^\infty(\Rn;\R^{n\times n})},
\end{equation}
and define
\begin{equation}\label{eq:S}\begin{aligned}
  S^+(x,r) &= \{\, s\in \Sn \colon |\sigma(x)\cdot s| > 2r\,\|\nabla \sigma\|_{L^\infty(\Rn;\R^{n\times n})} \,\},\\
  S^-(x,r) &= \Sn \setminus S^+(x,r),\\
  \hat{r}(x,s) &= \sup\{\, r\geq 0 \colon s\in S^+(x,r) \,\} = \frac{|\sigma(x)\cdot s|}{2\|\nabla \sigma\|_{L^\infty(\Rn;\R^{n\times n})}},\\
  O^+_{\delta,\epsilon}(x) &= \{\, z \in \Rn \colon \epsilon\leq |z|<\delta,
  z/|z|\in S^+(x,|z|) \,\}\\
  &= \{\, z \in \Rn \colon \epsilon\leq |z|<\min\{\delta, \hat{r}(x,z/|z|)\} \,\}, \quad\text{and}\\
  O^-_{\delta,\epsilon}(x) &= \{\, z \in \Rn \colon \epsilon\leq |z|<\delta,
  z/|z|\in S^-(x,|z|) \,\}.
\end{aligned}\end{equation}
Note that \(0\leq \hat{r}(x,s)\leq \hat{\delta}(x)/2\), for all \(x\in \Rn\) and \(s\in \Sn\).
In particular, \(O^+_{\delta,\epsilon}(x) \subset B(0,\hat{\delta}(x)/2)\).
We also note that \(O^+_{\delta,\epsilon}(x)  = -O^+_{\delta,\epsilon}(x)\).
The general idea is that \(O^-_{\delta,\epsilon}(x)\) is ``small,'' for small \(\delta\), while on \(O^+_{\delta,\epsilon}(x)\) Taylor's theorem applies to our integrand.

Let us start by establishing the uniform boundedness of \(J_{p,\delta,\epsilon}[\sigma]\).
\begin{lemma}\label{lem:Lip_plt2}
  Let \(\sigma\in C^1_c(\Rn;\Rn)\).
  Then there is \(L=L(\sigma,p,n)>0\) such that for all \(x\in \Rn\) and \(r> 0\) we have the inequality
  \begin{equation}\label{eq:Lip_plt2}
     \left|\int_{\Sn} j_{p,\sigma,s}(x+rs)\,\mathrm{d}s\right| \leq L r.
  \end{equation}
\end{lemma}
\begin{proof}
  We begin by recalling that
  \[\begin{aligned}
  \int_{\Sn} j_{p,\sigma,s}(x)\,\mathrm{d}s &= 0,
  \end{aligned}\]
  meaning that 
  \[\begin{aligned}
  \int_{\Sn} j_{p,\sigma,s}(x+rs)\,\mathrm{d}s &= \int_{\Sn} [j_{p,\sigma,s}(x+rs)-j_{p,\sigma,s}(x)]\,\mathrm{d}s.
  \end{aligned}\]
  We then split \(\Sn\) into two complementary, not necessarily nonempty parts:
  \begin{equation*}
    \begin{aligned}
    N &= \{\, s\in \Sn \colon |\sigma(x+ts)|>0, \forall t\in [0,1]\,\},\qquad\text{and} \\
    Z &= \{\, s\in \Sn \colon \exists \hat{t}_s \in [0,1] \text{\ such that\ } \sigma(x+\hat{t}_s s)=0\,\}.
    \end{aligned}
  \end{equation*}
  Arguing as in case 2 in the proof of Proposition~\ref{prop:i_lip}, we get the following estimate
  \begin{equation}\label{eq:I0}
    \int_{Z} [j_{p,\sigma,s}(x+rs)-j_{p,\sigma,s}(x)]\,\mathrm{d}s
    \leq |Z| \|\nabla \sigma\|_{L^\infty(\Rn;\Rn)}r.
  \end{equation}
  Thus we will now focus on the estimate on \(N\).
  With this in mind we further split \(N\) into
  \begin{equation*}
    \begin{aligned}
    N_1 &= \{\, s\in N \colon |\sigma(x+rs)|> |\sigma(x)| \},\qquad\text{and} \\
    N_2 &= \{\, s\in N \colon |\sigma(x+rs)|\leq |\sigma(x)|\,\}.
    \end{aligned}
  \end{equation*}
  Note that the integrand has the structure \(ab-cd\) with 
  \(a=|\sigma(x+rs)|^{2-p}\), \(b=|\sigma(x+rs)\cdot s|^{p-2}[\sigma(x+rs)\cdot s]\),
  \(c=|\sigma(x)|^{2-p}\), and \(d=|\sigma(x)\cdot s|^{p-2}[\sigma(x)\cdot s]\).
  On \(N_1\) we use the formula
  \(ab-cd = (a-c)(b-d) + c(b-d) + (a-c)d\), whereas on \(N_2\) we utilize the equality
  \(ab-cd = c(b-d) + (a-c)b\).
  We can now write
  \begin{equation*}\begin{aligned}
    &\int_{N} [j_{p,\sigma,s}(x+rs)-j_{p,\sigma,s}(x)]\,\mathrm{d}s
    \\&= 
    \int_{N_1} [|\sigma(x+rs)|^{2-p}-|\sigma(x)|^{2-p}|]\{|\sigma(x+rs)\cdot s|^{p-2}[\sigma(x+rs)\cdot s]-|\sigma(x)\cdot s|^{p-2}[\sigma(x)\cdot s]\}\,\mathrm{d}s
    \\&\quad+
    |\sigma(x)|^{2-p}\int_{N} \{|\sigma(x+rs)\cdot s|^{p-2}[\sigma(x+rs)\cdot s]-|\sigma(x)\cdot s|^{p-2}[\sigma(x)\cdot s]\}\,\mathrm{d}s
    \\&\quad+
    \int_{N_1} [|\sigma(x+rs)|^{2-p}-|\sigma(x)|^{2-p}]|\sigma(x)\cdot s|^{p-2}[\sigma(x)\cdot s]\,\mathrm{d}s
    \\&\quad+
    \int_{N_2} [|\sigma(x+rs)|^{2-p}-|\sigma(x)|^{2-p}]|\sigma(x+rs)\cdot s|^{p-2}[\sigma(x+rs)\cdot s]\,\mathrm{d}s    
    \\&= E_1+E_2+E_3+E_4.
  \end{aligned}\end{equation*}
  We will now estimate each of the terms \(E_1\), \(E_2\), \(E_3\), and \(E_4\).

  \smallskip{\bfseries Esimate for \(E_1\).}
  Utilizing the H{\"o}lder continuity of the function \(\R\ni t \mapsto |t|^{\alpha-1}t\)
  for \(\alpha \in (0,1)\) twice, first for $\alpha=2-p$, then for $\alpha=p-1$, we can estimate \(E_1\) using Taylor's theorem:
  \begin{equation}\label{eq:I1}\begin{aligned}
    |E_1| &\leq \int_{N_1} \big||\sigma(x+rs)|-|\sigma(x)|\big|^{2-p}\big|[\sigma(x+rs)\cdot s]-[\sigma(x)\cdot s]\big|^{p-1}\,\mathrm{d}s
    \\&\leq |N_1|\|\nabla\sigma\|^{(2-p)+(p-1)}_{L^\infty(\Rn;\Rn)}r^{(2-p)+(p-1)}
    =|N_1|\|\nabla\sigma\|_{L^\infty(\Rn;\Rn)}r.
  \end{aligned}\end{equation}

  \smallskip{\bfseries Esimate for \(E_2\).}
  To estimate \(E_2\) we can for example consider the sets
  \[\begin{aligned}
    S_1 &= \{\, s\in N \colon \sigma(x+rs)\cdot s >  |\sigma(x)\cdot s| \,\},\\
    S_2 &= \{\, s\in N \colon \sigma(x+rs)\cdot s < -|\sigma(x)\cdot s| \,\},\\
    S_3 &= \{\, s\in N \colon |\sigma(x+rs)\cdot s| \leq |\sigma(x)\cdot s| \,\}, \qquad\text{and}\\
    T &= \{\, s\in N \colon \sign[\sigma(x+rs)\cdot s] \neq \sign[\sigma(x)\cdot s]\,\}.
  \end{aligned}\]
  Additionally, we recall that a differentiable concave function $\phi:\R \to \R$ satisfies
  \begin{align*}
      \phi(t_1)-\phi(t_0) \leq \phi'(t_0)[t_1-t_0], \quad \forall (t_0,t_1) \in \R \times \R,
  \end{align*}
  and that the reverse inequality holds assuming convexity. Utilizing the monotonicity and concavity of the function \([0,\infty)\ni t\mapsto |t|^{p-2}t\) on \(S_{1}\), we have the estimate
  \begin{equation}\label{eq:est_u1}\begin{aligned}
    0\leq |\sigma(x+rs)\cdot s|^{p-2}[\sigma(x+rs)\cdot s]-|\sigma(x)\cdot s|^{p-1}
    &\leq
    (p-1)|\sigma(x)\cdot s|^{p-2}[\sigma(x+ rs)\cdot s-|\sigma(x)\cdot s|]
    \\&\leq
    (p-1)|\sigma(x)\cdot s|^{p-2}[\sigma(x+ rs)\cdot s-\sigma(x)\cdot s]
    \\&\leq (p-1)|\sigma(x)\cdot s|^{p-2} \|\nabla\sigma\|_{L^\infty(\Rn;\Rn)}r,
  \end{aligned}\end{equation}
  which is valid everywhere on \(S_1\) outside of the null-set \(\{\, s \in \Sn \colon \sigma(x)\cdot s = 0 \,\}\), which will become irrelevant after integration.
  Similarly, utilizing the monotonicity and convexity of the function \((-\infty,0]\ni t\mapsto |t|^{p-2}t\), on \(S_{2}\) we have the estimate
  \begin{equation}\label{eq:est_u2}\begin{aligned}
    0\geq |\sigma(x+rs)\cdot s|^{p-2}[\sigma(x+rs)\cdot s]+|\sigma(x)\cdot s|^{p-1}
    &\geq
    (p-1)|\sigma(x)\cdot s|^{p-2}[\sigma(x+ rs)\cdot s+|\sigma(x)\cdot s|]
    \\&\geq
    (p-1)|\sigma(x)\cdot s|^{p-2}[\sigma(x+ rs)\cdot s-\sigma(x)\cdot s]
    \\&\geq -(p-1)|\sigma(x)\cdot s|^{p-2} \|\nabla\sigma\|_{L^\infty(\Rn;\Rn)}r,
  \end{aligned}\end{equation}
  with the same comment as for the estimate for \(S_1\).
  We now focus on \(T\).  We note the bounds due to Taylor's theorem:
  \[
    \sigma(x)\cdot s - r\|\nabla\sigma\|_{L^\infty(\Rn;\Rn)}\leq \sigma(x+rs)\cdot s \leq \sigma(x)\cdot s + r\|\nabla\sigma\|_{L^\infty(\Rn;\Rn)}.
  \]
  Thus the sign difference between \(\sigma(x)\cdot s\) and \(\sigma(x+rs)\cdot s\) may only occur when \(s\) belongs to the spherical segment \(S^-(x,r)\), see~\eqref{eq:S}. Hence $T \subset S^-(x,r)$,
  whose measure is bounded by 
  \begin{equation}\label{eq:That}
    |S^-(x,r)| \leq 2\pi^{-1}|\mathbb{S}^{n-2}||\sigma(x)|^{-1}\|\nabla\sigma\|_{L^\infty(\Rn;\Rn)} r.
  \end{equation}

  Armed with~\eqref{eq:est_u1}, \eqref{eq:est_u2}, and~\eqref{eq:That} we can estimate the contribution to \(E_2\) coming from \(S_1\cup S_2\) by
  \begin{equation}\label{eq:I2_1}\begin{aligned}
    &|\sigma(x)|^{2-p}\left|\int_{S_1\cup S_2} \{|\sigma(x+rs)\cdot s|^{p-2}[\sigma(x+rs)\cdot s]-|\sigma(x)\cdot s|^{p-2}[\sigma(x)\cdot s]\}\,\mathrm{d}s\right|
    \\
    &\leq|\sigma(x)|^{2-p}\int_{S_1\cup S_2} \left||\sigma(x+rs)\cdot s|^{p-2}[\sigma(x+rs)\cdot s]-\sign[\sigma(x+rs)\cdot s]|\sigma(x)\cdot s|^{p-1}\right|\,\mathrm{d}s
    \\
    &\quad+|\sigma(x)|^{2-p}\int_{S_1\cup S_2}\left|\sign[\sigma(x+rs)\cdot s]|\sigma(x)\cdot s|^{p-1} -\sign[\sigma(x)\cdot s]|\sigma(x)\cdot s|^{p-1} \right|\,\mathrm{d}s
    \\
    &\leq
    |\sigma(x)|^{2-p}\int_{S_1\cup S_2} \big||\sigma(x+rs)\cdot s|^{p-2}[\sigma(x+rs)\cdot s]-\sign[\sigma(x+rs)\cdot s]|\sigma(x)\cdot s|^{p-1}\big|\,\mathrm{d}s\\
    &\quad+2|\sigma(x)|^{2-p}\int_{(S_1\cup S_2)\cap T} |\sigma(x)\cdot s|^{p-1}\,\mathrm{d}s
    \\
    &\leq
    (p-1)\|\nabla\sigma\|_{L^\infty(\Rn;\Rn)} r \int_{S_1\cup S_2}\left|\frac{\sigma(x)}{|\sigma(x)|}\cdot s\right|^{p-2}\,\mathrm{d}s 
    + 2|\sigma(x)||T| \sup_{s\in (S_1\cup S_2)\cap T}\left|\frac{\sigma(x)}{|\sigma(x)|}\cdot s\right|^{p-1} \\
    &\leq [(p-1) |\Sn|K_{p-2,n} + 4\pi^{-1}|\mathbb{S}^{n-2}|]\|\nabla\sigma\|_{L^\infty(\Rn;\Rn)} r.
  \end{aligned}\end{equation}

  Let us finally focus on contributions to \(E_2\) from \(S_3 = (S_3\cap S^-(x,r))\cup  (S_3\cap S^+(x,r))\).
  On \(S_3\cap S^-(x,r)\), utilizing the monotonicity of the function \((-\infty,0]\ni t\mapsto t^{p-1}\) we can simply write
  \begin{equation}\label{eq:I2_2}\begin{aligned}
    &|\sigma(x)|^{p-2}\left|\int_{S_3\cap S^-(x,r)} \{|\sigma(x+rs)\cdot s|^{p-2}[\sigma(x+rs)\cdot s]-|\sigma(x)\cdot s|^{p-2}[\sigma(x)\cdot s]\}\,\mathrm{d}s\right|
    \\
    &\leq
    2|\sigma(x)|^{2-p}\int_{S_3\cap S^-(x,r)} |\sigma(x)\cdot s|^{p-1}\,\mathrm{d}s\\
    &\leq 4\pi^{-1}|\mathbb{S}^{n-2}|\|\nabla\sigma\|_{L^\infty(\Rn;\Rn)} r.
  \end{aligned}\end{equation}

  Let \(s\in S^+(x,r)\) be  an arbitrary vector; in particular, \(0<r < \hat{r}(x,s)\).
  Let further \(t\in [0,1]\) and \(\tilde{r}\in [r,\hat{r}(x,s))\) be arbitrary.
  We note the inclusion \(s\in S^+(x,\tilde{r})\), which follows directly from the definition~\eqref{eq:S}, as well as the trivial inequality \(rt < \hat{r}(x,s)\).
  Consequently, utilizing~\eqref{eq:lbdot} we get the inequality
  \[|\sigma(x+rts)\cdot s|\geq \underbrace{|\sigma(x)\cdot s|}_{> 2\tilde{r}\|\nabla \sigma\|_{L^\infty(\Rn;\R^{n\times n})}}- rt\|\nabla \sigma\|_{L^\infty(\Rn;\R^{n\times n})}.\]
  By taking the supremum of the right hand side over \( \tilde{r}\in[r,\hat{r}(x,s))\)
  we arrive at the estimate
  \begin{equation}\label{eq:lbdot12}|\sigma(x+rts)\cdot s|
  \geq \hat{r}(x,s)\|\nabla \sigma\|_{L^\infty(\Rn;\R^{n\times n})}
  = \frac{1}{2}|\sigma(x)\cdot s|>0.\end{equation}
  Thus on \(S_3\cap S^+(x,r)\) we can apply Taylor's theorem to the integrand of \(E_2\) to get
  \begin{equation}\label{eq:I2_3}\begin{aligned}
    &|\sigma(x)|^{2-p}\left|\int_{S_3\cap S^+(x,r)} \{|\sigma(x+rs)\cdot s|^{p-2}[\sigma(x+rs)\cdot s]-|\sigma(x)\cdot s|^{p-2}[\sigma(x)\cdot s]\}\,\mathrm{d}s\right|
    \\
    &\leq
    |\sigma(x)|^{2-p}(p-1)r\int_{S_3\cap S^+(x,r)} \int_{0}^{1}
    |\sigma(x+rts)\cdot s|^{p-2}|s^T \nabla\sigma(x+rts)\cdot s|\,\mathrm{d}t\,\mathrm{d}s
    \\
    &\leq
    2^{2-p}(p-1)\|\nabla\sigma\|_{L^\infty(\Rn;\Rn)}r
    \int_{S_3\cap S^+(x,r)} \left|\frac{\sigma(x)}{|\sigma(x)|}\cdot s\right|^{p-2}\,\mathrm{d}s
    \\
    &\leq 2^{2-p}(p-1) |\Sn| K_{p-2,n} \|\nabla\sigma\|_{L^\infty(\Rn;\Rn)}r.
  \end{aligned}
  \end{equation}
  It remains to sum up the inequalities~\eqref{eq:I2_1}, \eqref{eq:I2_2}, and 
  \eqref{eq:I2_3} to obtain the desired estimate of \(E_2\).

  \smallskip{\bfseries Esimate for \(E_3\).}
  On \(N_1\), we have the following estimate, utilizing the monotonocity and concavity of \([0,\infty)\ni t \mapsto t^{2-p}\) and Taylor's theorem:
  \begin{equation*}
    \begin{aligned}
      0\leq |\sigma(x+rs)|^{2-p}-|\sigma(x)|^{2-p} 
      &\leq 
      (2-p)|\sigma(x)|^{1-p}[|\sigma(x+rs)|-|\sigma(x)|]
      \\&\leq
      (2-p)|\sigma(x)|^{1-p} \|\nabla\sigma\|_{L^\infty(\Rn;\Rn)}r.
    \end{aligned}
  \end{equation*}
  After integration, this yields:
  \begin{equation}\label{eq:I3}
    \begin{aligned}
      |E_3|&\leq (2-p)\|\nabla\sigma\|_{L^\infty(\Rn;\Rn)}r\int_{N_1} |\sigma(x)|^{1-p}|\sigma(x)\cdot s|^{p-1}\,\mathrm{d}s\\
      &\leq (2-p)|N_1|\|\nabla\sigma\|_{L^\infty(\Rn;\Rn)}r.
    \end{aligned}
  \end{equation}

  \smallskip{\bfseries Esimate for \(E_4\).}
  We proceed in the same way as with \(E_3\):
  \begin{equation*}
    \begin{aligned}
      0\leq |\sigma(x)|^{2-p}-|\sigma(x+rs)|^{2-p} 
      &\leq 
      (2-p)|\sigma(x+rs)|^{1-p}[|\sigma(x)|-|\sigma(x+rs)|]
      \\&\leq
      (2-p)|\sigma(x+rs)|^{1-p} \|\nabla\sigma\|_{L^\infty(\Rn;\Rn)}r.
    \end{aligned}
  \end{equation*}
  After integration, this yields:
  \begin{equation}\label{eq:I4}
    \begin{aligned}
      |E_4|&\leq (2-p)\|\nabla\sigma\|_{L^\infty(\Rn;\Rn)}r\int_{N_2} |\sigma(x+rs)|^{1-p}|\sigma(x+rs)\cdot s|^{p-1}\,\mathrm{d}s\\
      &\leq (2-p)|N_2|\|\nabla\sigma\|_{L^\infty(\Rn;\Rn)}r.
    \end{aligned}
  \end{equation}

  Summing up the estimates \eqref{eq:I0}, \eqref{eq:I1},  \eqref{eq:I2_1},  \eqref{eq:I2_2},  \eqref{eq:I2_3},  \eqref{eq:I3},  and \eqref{eq:I4}, and collecting all the constants into \(L\) concludes the proof of the claim.
\end{proof}

\begin{corollary}
  Assume that \(\sigma\in C^1_c(\Rn;\Rn)\).
  Then we have the inequality
  \begin{equation}\label{eq:Idelta_bound_plt2}
    |J_{p,\delta,\epsilon}[\sigma](x)| \leq L(\sigma,p,n) |\Sn|^{-1} K_{p,n}^{-1}, \qquad \forall 0<\epsilon<\delta, x\in \Rn,
  \end{equation}
  where the constant \(L(\sigma,p,n)>0\) is as in Lemma~\ref{lem:Lip_plt2}.
\end{corollary}
\begin{proof}
  \[
    \begin{aligned}
  |J_{p,\delta,\epsilon}[\sigma](x)| &\leq 
  \int_{\epsilon}^{\delta}
  r^{n-1}\underbrace{\bigg|\int_{\Sn}
  |\sigma(x+rs)|^{2-p} |\sigma(x+rs)\cdot s|^{p-2}[\sigma(x+rs)\cdot s]\,\mathrm{d}s \bigg|}_{\leq L(\sigma,p,n)r} r^{p-1} \Wd^p(r)\,\mathrm{d}r
    \\&\leq
    L(\sigma,p,n) |\Sn|^{-1} K_{p,n}^{-1}.
  \end{aligned}\qedhere
  \]
\end{proof}

We can now proceed with the analysis of the pointwise behaviour of \(J_{p,\delta,\epsilon}[\sigma]\).

\begin{proposition}\label{prop:consistency_s}
  Let \(\sigma \in C^1_c(\Rn;\Rn)\), and let us fix \(x\in \Rn\) such that \(\sigma(x)\neq 0\).
  Then
  \begin{equation}\label{eq:cons_S}
    \lim_{\delta\searrow 0}\lim_{\epsilon\searrow 0} 
    \int_{O^+_{\delta,\epsilon}}\left|
    \int_0^1 
    \left| \frac{\sigma(x+tz)}{|\sigma(x+tz)|} \cdot \frac{z}{|z|}\right|^{p-2}\,\mathrm{d}t\,
    -
    \left| \frac{\sigma(x)}{|\sigma(x)|} \cdot \frac{z}{|z|}\right|^{p-2}\right|
    \,|z|^{p}\Wd^p(z)\,\mathrm{d}z = 0.
  \end{equation}
\end{proposition}
\begin{proof}
  Recall that we have the inequality~\eqref{eq:lbdot12} bounding the inner product 
  \(\sigma(x+tz)\cdot z/|z|\), whence also \(|\sigma(x+tz)|\), away from zero for \(z/|z| \in S^+(x,|z|)\), uniformly with respect to \(|z|\in [0,\hat{r}(x,z/|z|))\) and \(t\in [0,1]\).
  Therefore,  the upper bound 
  \begin{equation}\label{eq:bndpm2}\begin{aligned}
  \left| \frac{\sigma(x+rts)}{|\sigma(x+rts)|} \cdot s\right|^{p-2}
  &\leq 
  2^{2-p}|\sigma(x)\cdot s|^{p-2}
  \sup_{z \in B(0,\delta)}|\sigma(x+z)|^{2-p} 
  \\&\leq [2\|\sigma\|_{L^\infty(\Rn;\Rn)}]^{2-p}|\sigma(x)\cdot s|^{p-2} \in L^1(\Sn),
  \end{aligned}\end{equation}
  holds uniformly for \(0<r<\hat{r}(x,s)\) and \(t\in[0,1]\).
  Furthermore, for each \(s\in \Sn\) outside of the null set \(S_0=\{\, s\in \Sn \colon |\sigma(x)\cdot s|=0\,\}\),
  as soon as \(\delta < \hat{r}(x,s)\) we can apply Taylor's theorem to conclude that
  \[\begin{aligned}
    &\lim_{\epsilon\searrow 0}\int_{\epsilon}^{\delta} r^{n-1}
    \int_{0}^{1}\left| \frac{\sigma(x+rts)}{|\sigma(x+rts)|} \cdot s\right|^{p-2}\,\mathrm{d}t\,
    |r|^{p}\Wd^p(r)\,\mathrm{d}r 
    \\&
    =
    |\Sn|^{-1}\left| \frac{\sigma(x)}{|\sigma(x)|} \cdot s\right|^{p-2}
    \underbrace{\lim_{\epsilon\searrow 0}
    \int_{\epsilon}^{\delta} r^{n-1}|\Sn|
    |r|^{p}\Wd^p(r)\,\mathrm{d}r}_{=K_{p,n}^{-1}} + O(\delta)
    \\&
    =
    K_{p,n}^{-1}|\Sn|^{-1}\left| \frac{\sigma(x)}{|\sigma(x)|} \cdot s\right|^{p-2} + O(\delta),
  \end{aligned}\]
  where the constant depends on \(s\in \Sn\setminus S_0\) but is independent from small \(\delta>0\).
As a consequence, dominated Lebesgue convergence theorem applies to the family of functions
\[
  s\mapsto \lim_{\varepsilon\searrow 0}\int_{\epsilon}^{\min\{\delta,\hat{r}(x,s)\}}r^{n-1}\int_{0}^{1}
  \left| \frac{\sigma(x+rts)}{|\sigma(x+rts)|} \cdot s\right|^{p-2}\,\mathrm{d}t\,
  r^{p}\Wd^p(r)\,\mathrm{d}r,
\]
and we have
\[
  \lim_{\delta\searrow 0}\lim_{\epsilon\searrow 0}
  \int_{\Sn}  \left|
  \int_{\epsilon}^{\min\{\delta,\hat{r}(x,s)\}}r^{n-1}\int_{0}^{1}
  \left| \frac{\sigma(x+rts)}{|\sigma(x+rts)|} \cdot s\right|^{p-2}\,\mathrm{d}t\,
  r^{p}\Wd^p(r)\,\mathrm{d}r - 
  K_{p,n}^{-1}|\Sn|^{-1}\left| \frac{\sigma(x)}{|\sigma(x)|} \cdot s\right|^{p-2}
  \right|\,\mathrm{d}s = 0,
\]
which concludes the proof.
\end{proof}

\begin{proposition}\label{prop:ptwise_plt2}
  Assume that \(\sigma \in C^2_c(\Rn;\Rn)\).
  Then
  \begin{equation}\label{eq:cons_plt2_lim}
    \lim_{\delta\searrow 0}\lim_{\epsilon\searrow 0} J_{p,\delta,\epsilon}[\sigma](x) = \Div\sigma(x), \qquad \text{for almost all \(x\in \Rn\)}.
  \end{equation}
\end{proposition}
\begin{proof}
  Note that cases 2 and 3 in the proof of Proposition~\ref{prop:consistency_pgt2} apply without any change.  Therefore, we will focus on the remaining case \(\sigma(x)\neq 0\).
  We remind the reader of the notation~\eqref{eq:S}.

  We begin by dealing with the integral over \(O^-_{\delta,\epsilon}(x)\).
  Recall the upper estimate of the measure of \(S^-(x,r)\) provided by~\eqref{eq:That},
  which implies that 
  \[|O^-_{\delta,\epsilon}(x)| = \int_{\epsilon}^{\delta} |S^-(x,r)|\,\mathrm{d}r= O(\delta^2).\]
  Furthermore, 
  for each \(z\in O^-_{\delta,\epsilon}(x)\) we have the estimate
  \begin{equation}\label{eq:i_Sminus}\begin{aligned}
    |j_{p,\sigma,z/|z|}(x+z)|&=
    |\sigma(x+z)|^{2-p}|\sigma(x+z)\cdot z/|z||^{p-1}
    \\&\leq
    |\sigma(x+z)|^{2-p}\underbrace{\big||\sigma(x+z)\cdot z/|z||^{p-1}
    - |\sigma(x)\cdot z/|z||^{p-1}\big|}_{\leq 
    [|z|\|\nabla \sigma\|_{L^\infty(\Rn;\R^{n\times n})}]^{p-1}}
    \\&+ |\sigma(x+z)|^{2-p}\underbrace{|\sigma(x)\cdot z/|z||^{p-1}}_{\leq 
    [2|z|\|\nabla \sigma\|_{L^\infty(\Rn;\R^{n\times n})}]^{p-1}}
    =O(\delta^{p-1}),
  \end{aligned}\end{equation}
  where we have utilized triangle inequality, H{\"o}lder continuity and monotonicity of the map \(\alpha\mapsto |\alpha|^{p-1}\), and the upper estimate on \(\big|\sigma(x)\cdot z/|z|\big|\) for \(z/|z|\in S^-(x,|z|)\).
  Combining~\eqref{eq:That} and~\eqref{eq:i_Sminus} yields the estimate
  \[\begin{aligned}\bigg|\int_{O^-_{\delta,\epsilon}(x)}
  j_{p,\sigma,z/|z|}(x+z) |z|^{p-1}\Wd^p(z)\,\mathrm{d}z\bigg|
  &\leq \sup_{z\in O^-_{\delta,\epsilon}(x)}|j_{p,\sigma,z/|z|}(x+z)| 
  \int_{\epsilon}^{\delta} \underbrace{|S^-(x,r)|}_{O(r)} r^{p-1}\Wd^{p}(r)\,\mathrm{d}r
  \\&\leq O(\delta^{p-1}),
  \end{aligned}\]
  with the constant depending on \(p\), \(n\), \(|\sigma(x)|^{-1}\), \(\|\sigma\|_{L^\infty(\Rn;\Rn)}\), and \(\|\nabla\sigma\|_{L^\infty(\Rn;\R^{n\times n})}\).

  We now proceed to the integrals over \(O^+_{\delta,\epsilon}\).
  In view of the bound~\eqref{eq:lbdot12} for \(z\in O^+_{\delta,\epsilon}\) we can apply Taylor's theorem to \(j_{p,\sigma,z/|z|}\) in the vicinity of the line segment connecting \(x\) and \(x+z\). The resulting Taylor polynomial terms are given in~\eqref{eq:Taylor0}, \eqref{eq:Taylor11}, and \eqref{eq:Taylor12}.

  As in Proposition~\ref{prop:consistency_pgt2}, the zero order term integrates to zero over \(O^+_{\delta,\epsilon}(x)=-O^+_{\delta,\epsilon}(x)\):
  \[
      \int_{O^+_{\delta,\epsilon}(x)}
      j_{p,\sigma,z/|z|}(x) |z|^{p-1}\Wd^p(z)\,\mathrm{d}z = 0.
  \]
  One of the first order terms is treated similarly to the term \eqref{eq:Taylor11} in Proposition~\ref{prop:consistency_pgt2}:
  \begin{equation*}\begin{aligned}
    &\int_{O^+_{\delta,\epsilon}(x)}\int_{0}^1
    \underbrace{\left[\frac{\sigma(x+tz)^T}{|\sigma(x+tz)|}\nabla\sigma(x+tz)\frac{z}{|z|}\right]}_{\text{Lipschitz}}
    \underbrace{\left|\frac{\sigma(x+tz)}{|\sigma(x+tz)|}\cdot\frac{z}{|z|}\right|^{p-2}
      \left[\frac{\sigma(x+tz)}{|\sigma(x+tz)|}\cdot\frac{z}{|z|}\right]}_{\text{H{\"o}lder}}
      \,\mathrm{d}t|z|^{p}\Wd^p(z)\,\mathrm{d}z
      \\&=
      \int_{O^+_{\delta,\epsilon}(x)}
      \left[\frac{\sigma(x)^T}{|\sigma(x)|}\nabla\sigma(x)\cdot \frac{z}{|z|}\right]
        \left|\frac{\sigma(x)}{|\sigma(x)|}\cdot\frac{z}{|z|}\right|^{p-2}
        \left[\frac{\sigma(x)}{|\sigma(x)|}\cdot\frac{z}{|z|}\right]
        |z|^{p}\Wd^p(z)\,\mathrm{d}z + O(\delta^{p-1})
      \\&=
      \int\limits_{B(0,\delta)\setminus B(0,\epsilon)}
      \left[\frac{\sigma(x)^T}{|\sigma(x)|}\nabla\sigma(x)\cdot \frac{z}{|z|}\right]
        \left|\frac{\sigma(x)}{|\sigma(x)|}\cdot\frac{z}{|z|}\right|^{p-2}
        \left[\frac{\sigma(x)}{|\sigma(x)|}\cdot\frac{z}{|z|}\right]
        |z|^{p}\Wd^p(z)\,\mathrm{d}z + O(\delta^{p-1}),
  \end{aligned}\end{equation*}
  where in the last equality we have utilized the continuity of the Lebesgue integral and the estimate \(|O^-_{\delta,\epsilon}(x)|=O(\delta^2)\).

  The other first order term is more troublesome:
  \begin{equation*}\begin{aligned}
    &\int_{O^+_{\delta,\epsilon}(x)}
    \int_{0}^1
    \underbrace{\left|\frac{\sigma(x+tz)}{|\sigma(x+tz)|}\cdot\frac{z}{|z|}\right|^{p-2}}_{\leq [2\|\sigma\|_{L^\infty(\Rn;\Rn)}]^{2-p}|\sigma(x)\cdot z/|z||^{p-2}}
    \underbrace{\left[\frac{z^T}{|z|}\nabla\sigma(x+tz)\frac{z}{|z|}\right]}_{\text{Lipschitz}}
    \,\mathrm{d}t
    |z|^{p}\Wd^p(z)\,\mathrm{d}z
    \\&=\int_{O^+_{\delta,\epsilon}(x)}
    \int_{0}^1
    \left|\frac{\sigma(x+tz)}{|\sigma(x+tz)|}\cdot\frac{z}{|z|}\right|^{p-2}
    \left[\frac{z^T}{|z|}\nabla\sigma(x)\frac{z}{|z|}\right]
    \,\mathrm{d}t
    |z|^{p}\Wd^p(z)\,\mathrm{d}z + O(\delta)
    \\&=\int_{\Sn}
    \left[s^T\nabla\sigma(x)s\right]
    \int_{\epsilon}^{\min\{\hat{r}(x,s),\delta\}} r^{n-1}
    \int_{0}^1
    \left|\frac{\sigma(x+rts)}{|\sigma(x+rts)|}\cdot s\right|^{p-2}
    \,\mathrm{d}t\,
    r^{p}\Wd^p(r)\,\mathrm{d}r\,\mathrm{d}s + O(\delta),\\
  \end{aligned}\end{equation*}
  in view of the bound~\eqref{eq:bndpm2}.
  To the last integral we can apply Proposition~\ref{prop:consistency_s}, which does not provide us with an error estimate but is sufficient for establishing its convergence and  identifying its limit:
  \begin{equation*}
    \begin{aligned}
      \lim_{\delta\searrow 0}\lim_{\epsilon\searrow 0}\bigg\{
        &\int_{O^+_{\delta,\epsilon}(x)}
      \int_{0}^1
   \left|\frac{\sigma(x+tz)}{|\sigma(x+tz)|}\cdot\frac{z}{|z|}\right|^{p-2}
    \left[\frac{z^T}{|z|}\nabla\sigma(x+tz)\frac{z}{|z|}\right]
    \,\mathrm{d}t
    |z|^{p}\Wd^p(z)\,\mathrm{d}z
    \\-&\int_{B(0,\delta)\setminus B(0,\epsilon)}
    \left|\frac{\sigma(x)}{|\sigma(x)|}\cdot\frac{z}{|z|}\right|^{p-2}
    \left[\frac{z^T}{|z|}\nabla\sigma(x)\frac{z}{|z|}\right]
    |z|^{p}\Wd^p(z)\,\mathrm{d}z
      \bigg\}=0,
    \end{aligned}
  \end{equation*}
  where we have extended the integration domain of the last integral from \(O^+_{\delta,\epsilon}(x)\) to \(B(0,\delta)\setminus B(0,\epsilon)\) since \(|O^-_{\delta,\epsilon}(x)|=O(\delta^2)\).
  It remains to appeal to Corollary~\ref{cor:crazytrace} to conclude the proof.
\end{proof}

\subsection{Numerical verification}

Jupyter notebooks allowing one to numerically verify the equality asserted in Lemma~\ref{lem:crazytrace} or the pointwise convergence established in Propositions~\ref{prop:consistency_pgt2} and~\ref{prop:ptwise_plt2} are publicly available via GitHub under \url{https://github.com/anev-aau/nloc_dual_lapl_p}.

\section*{Acknowledgment}
 The authors are gratefult to Prof.~Horia Cornean for multiple fruitful discussions of the material in Section~\ref{sec:Fconsist}.

\bibliographystyle{elsarticle-num}
\bibliography{nloc_dual_lapl_p}

\begin{thebibliography}{10}
\expandafter\ifx\csname url\endcsname\relax
  \def\url#1{\texttt{#1}}\fi
\expandafter\ifx\csname urlprefix\endcsname\relax\def\urlprefix{URL }\fi
\expandafter\ifx\csname href\endcsname\relax
  \def\href#1#2{#2} \def\path#1{#1}\fi

\bibitem{cea1970example}
J.~C{\'e}a, K.~Malanowski, An example of a max-min problem in partial
  differential equations, SIAM Journal on Control 8~(3) (1970) 305--316.
\newblock \href {https://doi.org/10.1137/03080} {\path{doi:10.1137/03080}}.

\bibitem{cherkaev2012variational}
A.~Cherkaev, Variational methods for structural optimization, Vol. 140,
  Springer Science \& Business Media, 2012.
\newblock \href {https://doi.org/10.1007/978-1-4612-1188-4}
  {\path{doi:10.1007/978-1-4612-1188-4}}.

\bibitem{allaire2012shape}
G.~Allaire, Shape optimization by the homogenization method, Vol. 146, Springer
  Science \& Business Media, 2012.
\newblock \href {https://doi.org/10.1007/978-1-4684-9286-6}
  {\path{doi:10.1007/978-1-4684-9286-6}}.

\bibitem{bendsoe2013topology}
M.~P. Bendsoe, O.~Sigmund, Topology optimization: theory, methods, and
  applications, Springer Science \& Business Media, 2013.
\newblock \href {https://doi.org/10.1007/978-3-662-05086-6}
  {\path{doi:10.1007/978-3-662-05086-6}}.

\bibitem{evgrafov2015chebyshev}
A.~Evgrafov, On {C}hebyshev's method for topology optimization of {S}tokes
  flows, Structural and Multidisciplinary Optimization 51~(4) (2015) 801--811.
\newblock \href {https://doi.org/10.1007/s00158-014-1176-x}
  {\path{doi:10.1007/s00158-014-1176-x}}.

\bibitem{kocvara2016}
M.~Ko\v{c}vara, S.~Mohammed, Primal-dual interior point multigrid method for
  topology optimization, SIAM Journal on Scientific Computing 38~(5) (2016)
  B685--B709.
\newblock \href {https://doi.org/10.1137/15M1044126}
  {\path{doi:10.1137/15M1044126}}.

\bibitem{kocvara2020}
A.~Brune, M.~Ko\v{c}vara, On barrier and modified barrier multigrid methods for
  three-dimensional topology optimization, SIAM Journal on Scientific Computing
  42~(1) (2020) A28--A53.
\newblock \href {https://doi.org/10.1137/19M1254490}
  {\path{doi:10.1137/19M1254490}}.

\bibitem{ioannis:2021}
I.~P.~A. Papadopoulos, P.~E. Farrell, T.~M. Surowiec, Computing multiple
  solutions of topology optimization problems, SIAM Journal on Scientific
  Computing 43~(3) (2021) A1555--A1582.
\newblock \href {https://doi.org/10.1137/20M1326209}
  {\path{doi:10.1137/20M1326209}}.

\bibitem{evgrafov2021nonlocal}
A.~Evgrafov, J.~C. Bellido, Nonlocal basis pursuit: Nonlocal optimal design of
  conductive domains in the vanishing material limit, SIAM J. Math. Analysis 55
  (2023) 2740--2773.
\newblock \href {https://doi.org/10.1137/22M1479166}
  {\path{doi:10.1137/22M1479166}}.

\bibitem{SiLe10}
S.~A. Silling, R.~B. Lehoucq, Peridynamic theory of solid mechanics, in:
  H.~Aref, E.~van~der Giessen (Eds.), Advances in Applied Mechanics, Vol.~44 of
  Advances in Applied Mechanics, Elsevier, 2010, pp. 73--168.
\newblock \href {https://doi.org/10.1016/S0065-2156(10)44002-8}
  {\path{doi:10.1016/S0065-2156(10)44002-8}}.

\bibitem{BuVal}
C.~Bucur, E.~Valdinoci, Nonlocal diffusion and applications, Vol.~20 of Lecture
  Notes of the Unione Matematica Italiana, Springer, [Cham]; Unione Matematica
  Italiana, Bologna, 2016.
\newblock \href {https://doi.org/10.1007/978-3-319-28739-3}
  {\path{doi:10.1007/978-3-319-28739-3}}.

\bibitem{AnMaRoTo10}
F.~Andreu-Vaillo, J.~M. Maz{\'o}n, J.~D. Rossi, J.~J. Toledo-Melero, Nonlocal
  diffusion problems, Vol. 165 of Mathematical Surveys and Monographs, American
  Mathematical Society, Providence, RI, 2010.

\bibitem{du}
Q.~Du, Nonlocal modeling, analysis, and computation, Vol.~94 of CBMS-NSF
  Regional Conference Series in Applied Mathematics, Society for Industrial and
  Applied Mathematics (SIAM), Philadelphia, PA, 2019.
\newblock \href {https://doi.org/10.1137/1.9781611975628.ch1}
  {\path{doi:10.1137/1.9781611975628.ch1}}.

\bibitem{madenci_oterkus}
E.~Madenci, E.~Oterkus, Peridynamic Theory and Its Applications, Springer,
  2014.
\newblock \href {https://doi.org/10.1007/978-1-4614-8465-3}
  {\path{doi:10.1007/978-1-4614-8465-3}}.

\bibitem{GiOs08}
G.~Gilboa, S.~Osher, Nonlocal operators with applications to image processing,
  Multiscale Model. Simul. 7~(3) (2008) 1005--1028.
\newblock \href {https://doi.org/10.1137/070698592}
  {\path{doi:10.1137/070698592}}.

\bibitem{Fife03}
P.~Fife, Some nonclassical trends in parabolic and parabolic-like evolutions,
  in: Trends in nonlinear analysis, Springer, Berlin, 2003, pp. 153--191.
\newblock \href {https://doi.org/10.1007/978-3-662-05281-5_3}
  {\path{doi:10.1007/978-3-662-05281-5_3}}.

\bibitem{CoCoElMa07}
C.~Cort{\'a}zar, J.~Coville, M.~Elgueta, S.~Mart{\'{\i}}nez, A nonlocal
  inhomogeneous dispersal process, J. Differential Equations 241~(2) (2007)
  332--358.
\newblock \href {https://doi.org/10.1016/j.jde.2007.06.002}
  {\path{doi:10.1016/j.jde.2007.06.002}}.

\bibitem{bourgain2001another}
J.~Bourgain, H.~Brezis, P.~Mironescu, Another look at {S}obolev spaces, in:
  J.~Menaldi, E.~Rofman, A.~Sulem (Eds.), Optimal Control and Partial
  Differential Equations: A volume in honor of {A}. {B}ensoussan's 60th
  birthday, IOS Press, Amsterdam, 2001.

\bibitem{ponce2004estimate}
A.~C. Ponce, An estimate in the spirit of {P}oincar{\'e}'s inequality, Journal
  of the European Mathematical Society 6~(1) (2004) 1--15.
\newblock \href {https://doi.org/10.4171/JEMS/1} {\path{doi:10.4171/JEMS/1}}.

\bibitem{ponce2004new}
A.~C. Ponce, A new approach to {S}obolev spaces and connections to
  $\gamma$-convergence, Calc. Var. Partial Differential Equations 19~(3) (2004)
  229--255.
\newblock \href {https://doi.org/10.1007/s00526-003-0195-z}
  {\path{doi:10.1007/s00526-003-0195-z}}.

\bibitem{fractional_laplacian}
M.~Daoud, E.~H. Laamri, Fractional {L}aplacians : A short survey, Discrete \&
  Continuous Dynamical Systems - S (2021).
\newblock \href {https://doi.org/10.3934/dcdss.2021027}
  {\path{doi:10.3934/dcdss.2021027}}.

\bibitem{hindsraduplap}
B.~Hinds, P.~Radu, {D}irichlet's principle and wellposedness of solutions for a
  nonlocal p-{L}aplacian system, Applied Mathematics and Computation 219~(4)
  (2012) 1411--1419.
\newblock \href {https://doi.org/10.1016/j.amc.2012.07.045}
  {\path{doi:10.1016/j.amc.2012.07.045}}.

\bibitem{gunzburgercalculus}
M.~Gunzburger, R.~B. Lehoucq, A nonlocal vector calculus with application to
  nonlocal boundary value problems, Multiscale Modeling \& Simulation 8~(5)
  (2010) 1581--1598.
\newblock \href {https://doi.org/10.1137/090766607}
  {\path{doi:10.1137/090766607}}.

\bibitem{bellidoexistence}
J.~Bellido, C.~Mora-Corral, Existence for nonlocal variational problems in
  peridynamics, SIAM Journal on Mathematical Analysis 46 (01 2014).
\newblock \href {https://doi.org/10.1137/130911548}
  {\path{doi:10.1137/130911548}}.

\bibitem{andres2015nonlocal}
F.~Andr{\'e}s, J.~Mu{\~n}oz, Nonlocal optimal design: a new perspective about
  the approximation of solutions in optimal design, Journal of Mathematical
  Analysis and Applications 429~(1) (2015) 288--310.
\newblock \href {https://doi.org/10.1016/j.jmaa.2015.04.026}
  {\path{doi:10.1016/j.jmaa.2015.04.026}}.

\bibitem{andres2017convergence}
F.~Andr{\'e}s, J.~Mu{\~n}oz, On the convergence of a class of nonlocal elliptic
  equations and related optimal design problems, Journal of Optimization Theory
  and Applications 172~(1) (2017) 33--55.
\newblock \href {https://doi.org/10.1007/s10957-016-1021-z}
  {\path{doi:10.1007/s10957-016-1021-z}}.

\bibitem{evgrafov2019non}
A.~Evgrafov, J.~C. Bellido, Non-local control in the conduction coefficients:
  well posedness and convergence to the local limit, SIAM Journal on Control
  and Optimization 58~(4) (2020) 1769--1794.
\newblock \href {https://doi.org/10.1137/19M126181X}
  {\path{doi:10.1137/19M126181X}}.

\bibitem{AnMuRo21}
F.~Andr{\'e}s, J.~Mu{\~n}oz, J.~Rosado, Optimal design problems governed by the
  nonlocal {$p$}-{L}aplacian equation, Math. Control Relat. Fields 11~(1)
  (2021) 119--141.
\newblock \href {https://doi.org/10.3934/mcrf.2020030}
  {\path{doi:10.3934/mcrf.2020030}}.

\bibitem{andres2023minimization}
F.~Andr{\'e}s, D.~Casta{\~n}o, J.~Mu{\~n}oz, Minimization of the compliance
  under a nonlocal p-{L}aplacian constraint, Mathematics 11~(7) (2023) 1679.
\newblock \href {https://doi.org/10.3390/math11071679}
  {\path{doi:10.3390/math11071679}}.

\bibitem{munoz2021generalized}
J.~Mu{\~n}oz, Generalized {P}once's inequality, Journal of Inequalities and
  Applications 2021~(1) (2021) 1--10.
\newblock \href {https://doi.org/10.1186/s13660-020-02543-1}
  {\path{doi:10.1186/s13660-020-02543-1}}.

\bibitem{JulioMunoz2021CtGP}
J.~Mu{\~n}oz, Correction to: Generalized {P}once's inequality, Journal of
  inequalities and applications 2021~(1) (2021) 1--5.
\newblock \href {https://doi.org/10.1186/s13660-021-02609-8}
  {\path{doi:10.1186/s13660-021-02609-8}}.

\bibitem{evgrafov2021dual}
A.~Evgrafov, J.~C. Bellido, The nonlocal {K}elvin principle and the dual
  approach to nonlocal control in the conduction coefficients (2021).
\newblock \href {https://doi.org/10.48550/ARXIV.2106.06031}
  {\path{doi:10.48550/ARXIV.2106.06031}}.

\bibitem{brezis2010functional}
H.~Brezis, Functional analysis, {S}obolev spaces and partial differential
  equations, Springer Science \& Business Media, 2010.
\newblock \href {https://doi.org/10.1007/978-0-387-70914-7}
  {\path{doi:10.1007/978-0-387-70914-7}}.

\bibitem{kurdila2006convex}
A.~J. Kurdila, M.~Zabarankin, Convex functional analysis, Springer Science \&
  Business Media, 2006.
\newblock \href {https://doi.org/10.1007/3-7643-7357-1}
  {\path{doi:10.1007/3-7643-7357-1}}.

\bibitem{bonnans2013perturbation}
J.~F. Bonnans, A.~Shapiro, Perturbation analysis of optimization problems,
  Springer Science \& Business Media, 2013.
\newblock \href {https://doi.org/10.1007/978-1-4612-1394-9}
  {\path{doi:10.1007/978-1-4612-1394-9}}.

\bibitem{girault_raviart}
V.~Girault, P.-A. Raviart, Finite element methods for {N}avier--{S}tokes
  equations, Vol.~5 of Springer Series in Computational Mathematics,
  Springer-Verlag, Berlin, 1986, theory and algorithms.
\newblock \href {https://doi.org/10.1007/978-3-642-61623-5}
  {\path{doi:10.1007/978-3-642-61623-5}}.

\bibitem{braides2006handbook}
A.~Braides, A handbook of {$\Gamma$}-convergence, in: Handbook of Differential
  Equations: stationary partial differential equations, Vol.~3, Elsevier, 2006,
  pp. 101--213.
\newblock \href {https://doi.org/10.1016/S1874-5733(06)80006-9}
  {\path{doi:10.1016/S1874-5733(06)80006-9}}.

\end{thebibliography}

\end{document}